\documentclass{amsart}
\usepackage{amsmath,amsthm,amsfonts, amssymb,amscd}
\usepackage[all]{xy}

\newtheorem{theorem}{Theorem}[section]
\newtheorem{lemma}[theorem]{Lemma}
\newtheorem{example}[theorem]{Example}
\newtheorem{corollary}[theorem]{Corollary}
\newtheorem{proposition}[theorem]{Proposition}

\newcommand{\lex}{\,\overrightarrow{\times}\,}

\newcommand{\Ker}{\mbox{\rm Ker}}

\newcommand{\Rad}{\mbox{\rm Rad}}

\newcommand{\Infinit}{\mbox{\rm Infinit}}

\newcommand{\RDP}{\mbox{\rm RDP}}

\begin{document}
\title[Lexicographic Effect Algebras]{ Lexicographic Effect Algebras}
\author[Anatolij Dvure\v{c}enskij]{Anatolij Dvure\v{c}enskij$^{1,2}$}
\date{}%
\maketitle
\begin{center}  \footnote{


This work was supported by  the Slovak Research and Development Agency under contract APVV-0178-11,  grant VEGA No. 2/0059/12 SAV, and
CZ.1.07/2.3.00/20.0051.
 }
Mathematical Institute,  Slovak Academy of Sciences,\\
\v Stef\'anikova 49, SK-814 73 Bratislava, Slovakia\\
$^2$ Depart. Algebra  Geom.,  Palack\'{y} University\\
17. listopadu 12, CZ-771 46 Olomouc, Czech Republic\\

E-mail: {\tt dvurecen@mat.savba.sk}
\end{center}

\begin{abstract}
In the paper we investigate a class of effect algebras which can be represented in the form of the lexicographic product $\Gamma(H\lex G,(u,0))$, where $(H,u)$ is an Abelian unital po-group and $G$ is an Abelian directed po-group. We study algebraic conditions when an effect algebra is of this form. Fixing a unital po-group $(H,u)$, the category of strong $(H,u)$-perfect effect algebra is introduced and it is shown that it is categorically equivalent to the category of directed po-group with interpolation. We show some representation theorems including a subdirect product representation by antilattice lexicographic effect algebras.
\end{abstract}

\section{Introduction}

A fundamental paper of \cite{FoBe} has introduced effect algebras as partial algebras with a primary operation $+$, where $a+b$ denotes the disjunction of two mutually excluding events $a$ and $b.$ The main idea was to describe algebraically an appropriate model for the so-called POV-measures (positive operator valued measures) in the effect algebra $\mathcal E(H)$ of Hermitian operators between the zero and the identity operators of a real, complex or quaternionic Hilbert space $H.$ In the last two decades, they describe an important class of so-called quantum structures which generalize Boolean algebras, orthomodular posets and orthomodular lattices, and orthoalgebras. For more information on effect algebras we recommend the book \cite{DvPu}. Effect algebras are studying algebras with two-valued reasoning as well as with many-valued features, where also some ideas of fuzzy approaches are used as well. Today effect algebras are an important tool for modeling processes of quantum mechanical measurement.

An important property of effect algebras is the Riesz Decomposition Property (RDP for short) which roughly speaking means an ability to perform a refinement of any two finite decompositions. It allows to represent an effect algebra as an interval in an Abelian unital po-group with interpolation (equivalently with RDP).
Nevertheless RDP fails for $\mathcal E(H)$, $\mathcal E(H)$ can be covered by blocks such that every block is an effect algebra with RDP, cf. \cite{Pul}.
An important class of effect algebras is constructed from MV-algebras. They are distributive lattices with RDP, and every two elements $a,b$ are compatible, i.e., there are three elements $a_1,b_1,c$ such that $a=a_1+c$, $b=b_1+c$ and $a_1+b_1+c$ is defined.

Effect algebras are not necessarily lattice ordered, but in important cases, they are antilattices, i.e. only comparable elements have $a\vee b$ and $a\wedge b$. The antilattices are generalizations of linearly ordered structures. In view of the famous Kadison's Antilattice Theorem \cite[Thm 58.4]{LuZa} we have: if $\mathcal B(H)$ is the set of all Hermitian operators of a Hilbert space $H$, then $\mathcal B(H)$ is an antilattice po-group whenever $H$ is a complex Hilbert space of dimension at least 2.

The aim of the paper is to study effect algebras of the form $\Gamma(H\lex G,(u,0))$, where $(H,u)$ is an Abelian unital po-group with RDP and $G$ is a directed po-group with RDP, we call then also lexicographic effect algebras. Roughly speaking, every effect algebra with RDP is of this form while if $E=\Gamma(H,u)$, then $E \cong \Gamma(H\lex G,(u,0))$, where $G=O$ is the null group. For us it will be interesting to exhibit situations when $G$ is not necessarily the zero group.

This problem was firstly investigated for MV-algebra in \cite{DiLe1} introducing so-called perfect MV-algebras; they are represented in the form $\Gamma(\mathbb Z\lex G,(1,0))$, where $\mathbb Z$ is the group of integers, and $G$ is an $\ell$-group. Perfect effect algebras are of the form $\Gamma(\mathbb Z\lex G,(1,0))$, $G$ is a directed po-group with RDP, and they were studied in \cite{177}. $n$-perfect pseudo MV-algebras were studied in \cite{Dv08} in the form $\Gamma(\frac{1}{n}\mathbb Z \lex G,(1,0))$. Recently, there appeared a series of papers where similar structures connected with the lexicographic product were studied, see e.g. \cite{DFL, 264, DvKr, DvKo, 264, DXY}.

In the paper we concentrate to exhibit algebraic conditions when effect algebras are intervals in lexicographic product of two po-groups. For this aim, we fix a unital po-group $(H,u)$ with RDP and introduce $(H,u)$-perfect effect algebras and strongly $(H,u)$-perfect ones. The latter ones are exactly of the form of a lexicographic product. We show that the category of strong $(H,u)$-perfect effect algebras with RDP are categorically equivalent with the category of directed po-groups with RDP. We show an important role of lexicographic ideals for our task. We describe cases when our $(H,u)$-perfect algebras are antilattices and when they are a subdirect product of antilattice lexicographic effect algebras.

The paper is organized as follows. Section 2 gathers the basic notions and results on effect algebras and po-groups. Section 3 describes lexicographic product, lexicographic effect algebras, and lexicographic ideals. In Section 4, we define $(H,u)$-perfect effect algebras and in Section 5, a stronger form of them, strong $(H,u)$-perfect effect algebras. A representation of strong $(H,u)$-perfect effect algebras in the form of lexicographic product is proved here. We show also the categorical equivalence of the category of strong $(H,u)$-perfect effect algebras with RDP to the category directed po-groups with RDP. In Section 6, we present some representation theorems starting with local effect algebras with strict and retractive ideal and finishing with a representation of lexicographic effect algebra with RDP as a subdirect product of antilattice lexicographic ideals with RDP.

\section{The Elements of Effect Algebras}

Following to \cite{FoBe}, we say that an {\it effect algebra} is  a partial algebra $(E;+,0,1)$ with a partially defined operation $+$ and with two constant
elements $0$ and $1$  such that, for all $a,b,c \in E$,
\begin{enumerate}

\item[(i)] $a+b$ is defined in $E$ if and only if $b+a$ is defined, and in
such a case $a+b = b+a;$

 \item[(ii)] $a+b$ and $(a+b)+c$ are defined if and
only if $b+c$ and $a+(b+c)$ are defined, and in such a case $(a+b)+c
= a+(b+c);$

 \item[(iii)] for any $a \in E$, there exists a unique
element $a^- \in E$ such that $a+a^-=1;$

 \item[(iv)] if $a+1$ is defined in $E$, then $a=0.$
\end{enumerate}

If we define $a \le b$ if and only if there exists an element $c \in
E$ such that $a+c = b$, then $\le$ is a partial ordering on $E$, and
we write $c:=b-a;$ then $a^- = 1 - a$ for any $a \in E.$ As a basic source of information about effect algebras we can recommend the monograph \cite{DvPu}. An effect algebra is not necessarily a lattice ordered set.

Let $E$ and $F$ be two effect algebras. A mapping $h:E\to F$ such that (i) $h(1)=1$, and (ii) if $a+b$ is defined in $E$, then $h(a)+h(b)$ is defined in $F$ and $h(a+b)=h(a)+h(b)$ is said to be a {\it homomorphism}. A bijective mapping $h:E\to F$ such that $h$ and $h^{-1}$ are homomorphisms is said to be an {\it isomorphism}; $E$ and $F$ are {\it isomorphic} and we write $E \cong F$.

A subset $F$ of $E$ is an {\it effect subalgebra} of $E$ if (i) $1\in F$, (ii) if $a \in F$, then $a^-\in F$, and (iii) if $a,b \in F$ and $a+b$ is defined in $E$, then $a+b \in F$.

A large class of effect algebras are intervals in Abelian po-groups. We remind that a {\it po-group} is a group $(G;+,0)$ written additively and endowed with a partial order $\le$ such that if $g\le h$, then $k+g+f\le k+h+f$ for all $k,f \in G$. We denote by $G^+=\{g\in G: g \ge 0\}$  and $G^-=\{g \in G: g\le 0\}$ the positive cone and negative cone, respectively, of $G$. If $G$ is a lattice under the partial order $\le$, we call it an $\ell$-{\it group} (or a {\it lattice ordered group}).
An element $u \in G^+$ is a {\it strong unit} of a po-group $G$ if, for any $g \in G$, there is an integer $n \ge 0$ such that $g\le nu.$ The pair $(G,u),$ where $u$ is a fixed strong unit of $G,$ is said to be a {\it unital po-group}. We will deal mainly with Abelian po-groups. For more information on po-groups and $\ell$-groups we recommend to consult the books \cite{Fuc, Gla, Goo}.

Now let $u$ be an element of $G^+$ and set $\Gamma(G,u)=[0,u]:=\{g\in G: 0 \le g \le u\}$. Then $(\Gamma(G,u);+,0,u)$ is an effect algebra, where $a+b$ is defined for $a,b \in \Gamma(G,u)$ iff $a\le u-b$, and in such a case, $a+b$ is equal to the original group addition.  An effect algebra $E$ is said to be an {\it interval effect algebra} if there is an Abelian unital po-group $(G,u)$ such that $E\cong \Gamma(G,u)$. We remind that a po-group $H$ is {\it Archimedean} provided that whenever $x,y \in H$ such that $nx \le y$ for all integers $n \ge 1$, then $x\le 0$.

We note that an Abelian po-group $(G;\le)$ is an {\it interpolation group} (or $G$ has the {\it Riesz Interpolation Property} (RIP for short), or simply $G$ has interpolation), whenever for $g_1,g_2, h_1,h_2 \in G$ such that $g_1,g_2 \le h_1, h_2$, there is an element $f \in G$ such that $g_1,g_2\le f \le h_1,h_2$. By \cite[Prop 2.1]{Goo}, an equivalent definition is, for $a_1,a_2,b_1,b_2 \in G^+$ such that $a_1+a_2=b_1+b_2$, there are four elements $c_{11},c_{12},c_{21},c_{22}\in G^+$ such that $a_1 = c_{11}+c_{12},$ $a_2= c_{21}+c_{22},$ $b_1= c_{11} + c_{21}$ and $b_2= c_{12}+c_{22}$; we call this property the {\it Riesz Decomposition Property} and we write RDP for short. If in the definition of RDP for po-groups, we change $G^+$ to an effect algebra $E$, we say that $E$ satisfies the RDP.

For non-commutative po-groups, in \cite{DvVe1, DvVe2}, there are introduced other types of the Riesz Decomposition Property. We say that a po-group $G$ satisfies (i) RDP$_1$ if it satisfies RDP and, for $c_{12}$ and $c_{21}$, we have $0\le x\le c_{12}$ and $0\le y \le c_{21}$ imply  $x+y=y+x$; (ii) RDP$_2$, if it satisfies RDP and, for $c_{12}$ and $c_{21}$, we have $c_{12}\wedge c_{21}=0.$

Important linearly ordered groups $\mathbb R$, the group of reals, and $\mathbb Z$, the group of integers, have RDP.

The basic representation theorem of effect algebras is the following result by \cite{Rav}:

\begin{theorem}\label{th:2.1}
If $ E$ is an effect algebra with \RDP, there exists a unique Abelian unital po-group with interpolation $(G,u)$ (up to isomorphism of unital po-groups)  such that $ E\cong  {\Gamma}(G,u).$
\end{theorem}

If we denote by $\mathcal{ULG}$ the category of Abelian unital po-groups whose objects are unital po-groups with RDP and  morphisms are homomorphisms of unital po-groups (i.e. order preserving homomorphisms that preserve fixed strong units). Similarly, let $\mathcal{EA}_{RDP}$ be the category of effect algebras with RDP whose objects are effect algebras with RDP and morphisms are homomorphisms of effect algebras. Then the functor $\Gamma: \mathcal{ULG} \to \mathcal{EA}_{RDP}$ defined by $(G,u)\mapsto \Gamma(G,u)$ and $\Gamma(h)=h|_{\Gamma(G,u)}$, where $h$ is a morphism of unital po-groups, defines a categorical equivalence of $\mathcal{ULG}$ and $\mathcal{EA}_{RDP}$, \cite[Thm 5.8]{177}.

Not only effect algebras with RDP are interval effect algebras. For example, let $\mathcal B(H)$ denote the po-group of Hermitian operators of a Hilbert space $H$ ordered by the property $A\le B$ iff $(Ax,x)\le (Bx,x)$ for every unit vector $x \in H$. Then the effect algebra $\mathcal E(H)=\Gamma(\mathcal B(H),I)$, where $I$ is the identity operator, is an interval effect algebra, but $\mathcal E(H)$ does not satisfy RDP.

Another important subclass of effect algebras with RDP is a class of MV-algebras. We recall that an {\it MV-algebra} is an
algebra $(M;\oplus,^*,0,1)$ of signature $\langle 2,1,0,0\rangle,$ where
$(M;\oplus,0)$ is a commutative monoid with neutral element $0$, and
for all $x,y \in M$
\begin{enumerate}
\item[(i)]  $(x^*)^*=x$;
\item[(ii)] $x\oplus 1 = 1$;
\item[(iii)] $1=0^*$;
\item[(iv)] $x\oplus (x\oplus y^*)^* = y\oplus (y\oplus x^*)^*.$
\end{enumerate}

Then $M$ is a distributive lattice.

On every MV-algebra $M$ we can define a partial operation, $+$, by $a+b$
is defined in $M$ iff $a\le b^*,$ and we set then $a+ b:= a\oplus
b.$ Then $(M;+,0,1)$ is an interval effect algebra with RDP,
moreover, thanks to \cite{Mun},  every MV-algebra is in fact an
interval $\Gamma(G,u),$ where $G$ is a unital $\ell$-group with $a^*=u-a$ and $a\oplus b:= (a+b)\wedge u,$ $a,b \in \Gamma(G,u).$ Conversely, every lattice ordered effect algebra $E$ with RDP can be converted into an MV-algebra, see e.g. \cite[Thm 8.8]{DvVe2}.

We note that in \cite{DvVe1, DvVe2}, there is introduced a non-commutative generalization of effect algebras called {\it pseudo effect algebras}, and in \cite{GeIo}, there is a non-commutative generalization of MV-algebras called {\it pseudo MV-algebras}. In both structures $+$ and $\oplus$, respectively, are not necessarily commutative. Pseudo effect algebras are sometimes also intervals in unital po-groups with RDP$_1$ not necessarily Abelian, and pseudo MV-algebras are always intervals in unital $\ell$-groups not necessarily Abelian, \cite{151}.

For every integer $n\ge 0$ and for every element $a \in E$, we set $0a=0$, $1a=a$, and $na=(n-1)a+a$ if $(n-1)a$ exists in $E$ and $(n-1)a\le a^-$. An element $a \in E$ is {\it infinitesimal} if $na$ is defined in $E$ for every integer $n \ge 1$. We denote by $\Infinit(E)$ the set of infinitesimals of $E$. Then (i) $0 \in \Infinit(E)$, $a\le b\in \Infinit(E)$ implies $a\in \Infinit(E)$, and (iii) $1 \notin \Infinit(E).$ An effect algebra is said to be {\it Archimedean} if $\Infinit(E)=\{0\}$. If a po-group $G$ is Archimedean, then $\Gamma(G,u)$ is an Archimedean effect algebra.

An {\it ideal} of an effect algebra $E$ is any nonempty subset $I$ of $E$ such that (i) $a\le b \in I$ implies $a \in I,$ and (ii) if $a,b \in I$ and $a+b$ is defined in $E$,  then $a + b \in I.$ Let $\mathcal I(E)$ be the set of ideals of an effect algebra $E$. An ideal $I$ of $E$ is a {\it Riesz ideal}, provides if $i\in I,a,b\in E$ and  $a+ b$ exists, $i\le
a+ b,$ then there exist $i_a,i_b\in I$ such that $ i_a \le a,$
$i_b\le b$ and $i\le i_a+ i_b$. It is clear that if $E$ satisfies RDP, then every ideal of $E$ is a Riesz ideal.

Given a subset $A$ of $E$, there is an ideal, $I_0(A)$, of $E$ generated by $A$. If $A=\{a\}$ is a singleton, then $I_0(a)$ simply defines the ideal generated by $a$. If $E$ satisfies RDP, then by \cite[Prop 3.2]{177},
$$
I_0(a)=\{x \in E: \exists \ a^0_i \in E, a_i^0\le a,\ i=1,\ldots,k,\ k \in \mathbb N,
\ x =a_1^0+\cdots +a_k^0\}.\eqno(2.1)
$$

An ideal is said to be (i) {\it maximal} if $I\ne E$ and it is not a proper subset of another ideal $J \ne E$, and  (ii) {\it prime} if $I_0(x)\cap I_0(y) \subseteq I$ implies $x \in I$ or $y \in I$. We denote by $\mathcal M(E)$ and $\mathcal P(E)$ the set of maximal  ideals and prime ideals  of $E$, respectively.  An  effect algebra $E$ is {\it local} if there is a unique maximal ideal. We define $\Rad(E)=\bigcap\{I: I \in \mathcal M(I)\}$; then $\Rad(E)$ is an ideal of $E$. If $E$ satisfies the RDP, then by \cite[Prop 4.1]{177}
$$
\Infinit(E) \subseteq \Rad(E).\eqno(2.2)
$$
Let $E$ be an effect algebra. If $I$ is a Riesz ideal, then $I$ defines a congruence $\sim_I$ on $E$ as follows: $a\sim_I b$ iff $a-e=b-f$ for some $e,f \in I$ such that $e\le a$ and $f \le b$. Then by \cite[Cor 3.1.17]{DvPu}, $E/I:=\{x/I: x/I=\{y \in E: y\sim x, x \in E\}\}$ is an effect algebra, and if $E$ satisfies RDP, then $E/I$ is with RDP, too, cf. \cite[Prop 4.1]{185}.

An analogue of a probability measure on an effect algebra is a state. We remind that a state on an effect algebra $E$ is a mapping $s:E \to [0,1]$ is a {\it state} if (i) $s(1)=1$, and (ii) $s(a+b)=s(a)+s(b)$ whenever $a+b$ is defined in $E$. We note that not every effect algebra has a state, but due to \cite[Cor 4.4]{Goo}, every interval effect algebra admits at least one state. It is clear that the set $\Ker(s):=\{a\in E: s(a)\}$ is an ideal of $E$. Similarly, if $f:E\to F$ is a homomorphism, then $\Ker(f)=\{a\in E: f(a)=0\}$ is an ideal.

We note that a {\it state} on a unital po-group $(G,u)$ is any homomorphism $s$ of unital po-groups from $(G,u)$ into $(\mathbb R,1)$ i.e. (i) $s(u)=1$, (ii) $s(g_1+g_2)=s(g_1)+s(g_2)$, and (iii) $s(g)\ge 0$ whenever $g\ge 0$.

A poset $E$ is an {\it antilattice} if only comparable elements of $E$ have
an infimum or a supremum. We note that every linearly ordered poset is an antilattice, and a lattice  is an antilattice iff it is linearly ordered.

In Example \ref{ex:6.1} below, there is an example of an antilattice effect algebra with RDP that is not a lattice.

Antilattices characterize prime ideals of $E$ because a proper ideal $I$ of an effect algebra $E$ with RDP is prime iff $E/I$ is an antilattice, cf. \cite[Prop 6.5]{177}. In addition, in view of \cite[Thm 6.17]{177}, an effect algebra $E=\Gamma(G,u)$, where $G$ is a po-group with the RDP, is an antilattice iff $G$ is an antilattice.

If $A,B$ are two nonempty subsets of $E,$ we set $A+B:=\{a+b: a \in A, b \in B, a+b \mbox{ is defined in }  E\}.$ We say that $A+B$ is {\it defined} in $E$ if $a+b$ is defined in $E$ for each $a \in A$ and each $b \in B.$  Similarly, we write $A \leqslant B$ if $a\le b$ for each $a \in A$ and each $b \in B.$ We denote by $A^-:=\{x^-\in E: x\in A\}$ and $x+A=\{x+a: a\in A$ if $x+a$ is defined in $A\}$.

\begin{lemma}\label{le:2.2}
Let $I$ be an ideal of an effect algebra $E$. Then the effect subalgebra of $E$ generated by $I$ is the set $\langle I\rangle =I\cup I^-$.
\end{lemma}

\begin{proof}
The statement is evident if $I=E$. Now let $I\ne E$.
Then $I\cap I^-=\emptyset$. Let $a,b \in I$. If $a+b \in E$, then $a+b \in I$. If $a+b^-\in E$, then $a+b^-+x=1=b+b^-$ for some $x \in E$,  so that $a+x=b$ which entails $x=(a+b^-)^-\in I$ and finally, $a+b^-\in I^-$. We note that $a^-+b^-$ is not defined in $E$. Indeed, otherwise, $a^-\le b$ giving $a^-\in I$ which is absurd. Hence, $\langle I\rangle$ is an effect subalgebra of $E$ generated by $I$.
\end{proof}

A poset $(A,\le)$ is (i) {\it upwards directed} provides $x,y \in A$, there is $z \in A$ such that $x\le z$ and $y \le z$, (ii) {\it downwards directed} provides $x,y \in A$, there is $z \in A$ such that $x\ge z$ and $y \ge z$, and (iii) {\it directed} if $A$ is both upwards and downwards directed. For example, every po-group is upwards directed iff it is downwards directed.

\section{Lexicographic Product and Lexicographic Pseudo Effect Algebras}

Let $G_1$ and $G_2$ be two po-groups and we define the  direct product group $G_1 \times G_2$ with the group addition defined by coordinates. We define the {\it lexicographic order} $\le$ on  $G_1 \times G_2$ by $(g_1,g_2) \le (h_1,h_2)$ iff either $g_1 < g_2$ or $g_1=g_2$ and $h_1 \le h_2,$ for $(g_1,g_2), (h_1,h_2) \in G_1 \times G_2,$ and $G_1 \lex G_2$ will denote the {\it lexicographic product} of $G_1$ and $G_2$ endowed with this defined lexicographic order.

We note that according to \cite[Thm 3.1]{DvKo} if $G_1$ is linearly ordered and $G_2$ satisfy RDP, then $G=G_1\lex G_2$ satisfies also RDP, by \cite[Thm 3.2]{DvKo}, $G$ satisfies RDP$_1$ whenever $G_1$ is linear and $G_2$ an Abelian directed po-group with RDP. Finally, due to \cite[Thm 3.8]{DvKo}, $G=\mathbb H \lex G_2$, where $\mathbb H$ is a subgroup of $\mathbb R$, $1\in \mathbb H$,  satisfies RDP$_1$ iff the directed po-group $G_2$ satisfies RDP$_1$. We note that $G$ satisfies RDP$_2$ iff $G$ is an $\ell$-group, \cite[Prop 4.2(ii)]{DvVe1}. Therefore,  $G_1\lex G_2$ satisfies RDP$_2$ iff $G_1$ is linearly ordered and $G_2$ is an $\ell$-group, cf. \cite[(d) p. 26]{Fuc}.

For Abelian po-groups $G_1$ and $G_2$, we have the following criterion \cite[Cor 2.12]{Goo}: $G_1\lex G_2$ satisfies RDP iff

\begin{enumerate}
\item[(i)] Both $G_1$ and $G_2$ satisfy RDP.
\item[(ii)] Either $G_1$ satisfies SRIP or $G_2$ is directed.
\end{enumerate}

We note that RDP and RDP$_1$ for Abelian po-groups coincide, but we do not know a complete answer to a question when $G_1\lex G_2$ does satisfy RDP$_1$ for non Abelian po-groups. The above conditions (i)--(ii) hold for po-groups that are not necessarily Abelian whenever $G_1\lex G_2$ has RDP$_1$, see \cite[Prop 3.2]{DvKr}.

Now let $(H,u)$ be a unital po-group and $G$ a po-group. Then $(H\lex G,(u,0))$ is again a unital po-group and
$$
\Gamma(H\lex G,(u,0))\eqno(2.2)
$$
is a pseudo effect algebra or even an effect algebra when both $H$ and $G$ are Abelian.

The aims of this paper is to study when an effect algebra is isomorphic to an effect algebra of the form (2.2). The first look to this problem shows that every interval effect algebra is of this kind. Indeed, let $E=\Gamma(H,u)$, where $(H,u)$ is an Abelian po-group. Then $E \cong \Gamma(H\lex O,(u,0))$, where $O$ is the zero group. On the other hand, $(O,0)$ is trivially a unital po-group. Then $\Gamma(O\lex G,(0,0))$ is a degenerate effect algebra, i.e. $0=1$.  We will not concentrate to such evident situations, and rather  we will deal mainly with the case when both $H$ and $G$ are nontrivial and with RDP.

We say that an effect algebra $E$ is {\it lexicographic} if there are an Abelian unital po-group $(H,u)$ and an Abelian directed po-group $G$ such that $E$ is isomorphic to $\Gamma(H\lex G,(u,0))$. This notion for MV-algebras was introduced in \cite{DFL} and for pseudo MV-algebras in \cite{275}; for them we have to assume that $H$ is linearly ordered and $G$ is an $\ell$-group.
We note that if $(H,u)=(\mathbb Z,1)$ and $G$ is Abelian, we have by (2.2) perfect effect algebras, \cite{177}, (for more details on perfect effect algebras see a note just before Theorem \ref{th:5.1} below) which generalize perfect MV-algebras studied in \cite{DiLe1}, where $G$ was an $\ell$-group. The case of $n$-perfect pseudo MV-algebras, i.e. the case $(H,u)=(\frac{1}{n}\mathbb Z,1)$ and $G$ an $\ell$-group, was studied in \cite{Dv08}, and the case $(H,u)=(\mathbb H,1)$, where $\mathbb H$ is a subgroup of the group $\mathbb R$, was investigated in \cite{264} for so-called strong $\mathbb H$-perfect pseudo effect algebras and in \cite{275} for strong $(H,u)$-perfect pseudo MV-algebras, where $(H,u)$ is a linearly ordered unital Abelian po-group.

We start with the following notion which was introduced for MV-algebras in \cite{CiTo} and for pseudo MV-algebras in \cite{275}. An ideal $I$ of an effect algebra $E$ with the RDP is said to be {\it retractive} if the canonical projection $\pi_I: E  \to E/I$ is retractive, i.e. there is a homomorphism $\delta_I: E/I \to E$ such that $\pi_I\circ \delta_I=id_{E/I}$. If an ideal $I$ is retractive, then $\delta_I$ is injective and $E/I$ is isomorphic to an effect subalgebra of $E$.

For example, if $E=\Gamma(H\lex G,(u,0))$ and $I=\{(0,g): g \in G^+\}$, where $(H,u)$ is an Abelian unital po-group and $G$ is a directed Abelian po-group both with RDP, then $I$ is an ideal, and due to  $E/I \cong \Gamma(H,u) \cong \Gamma(H\lex \{0\},(u,0)) \subseteq \Gamma(H\lex G,(u,0))$, $I$ is retractive.

The notion of a lexicographic ideal introduced in \cite{DFL} only for MV-algebras and in \cite{275} for pseudo MV-algebras will be now extended also for effect algebras. We say that an ideal $I$ of an effect algebra $E$ with RDP, $\{0\}\ne I \ne E$, is {\it lexicographic} if
\begin{enumerate}
\item[{\rm (i)}] $I$ is strict, i.e. $x/I<y/I$ implies $x<y$;
\item[{\rm (ii)}] $I$ is retractive;
\item[{\rm (iii)}] $I$ is prime.
\end{enumerate}
We denote by $\mathrm{LexId}(E)$ and by $\mathrm{StrId}(E)$ the set of lexicographic ideals of $E$ and strict nontrivial ideals, respectively, of $E$.

\begin{example}\label{ex:3.1}
We define three  effect algebras: $E_1=\Gamma(\mathbb Z \lex (\mathbb Z\lex \mathbb Z),(1,(0,0)))$, $E_2 = \Gamma((\mathbb Z \lex \mathbb Z)\lex \mathbb Z,((1,0),0))$, and $E=\Gamma(\mathbb Z \lex \mathbb Z\lex \mathbb Z,(1,0,0))$ which are  mutually isomorphic.  We see that $I_1=\{(0,0,n): n \ge 0\}$ and $I_2=\{(0,m,n): m > 0, n \in \mathbb Z \mbox{ or } m=0, n\ge 0\}$ are only two lexicographic ideals of $E$ and $I_1 \subset I_2$.
\end{example}

It is worthy of recalling that if $\Gamma(H_1\lex G_1,(u_1,0))\cong \Gamma(H_2\lex G_2,(u_2,0))$, it does not mean that $(H_1,u_1)\cong (H_2,u_2)$ and $G_1\cong G_2$. Indeed, in Example \ref{ex:3.1}, we have $E_1\cong E_2$ and the linear unital po-groups $(\mathbb Z,1)$ and $(\mathbb Z \lex \mathbb Z,(1,0))$ are not isomorphic while the first one is Archimedean and the second one is not Archimedean. Analogously, $G_1$ and $G_2$ are not isomorphic.

But when we look for a representation such that $(H,u)$ is fixed, then we have uniqueness of $G$ in (2.2).

We note that it can happen that $\mathrm{LexId}(E)$ is an empty set. Indeed, let $E =\Gamma(\mathbb Z \lex \mathbb Z,(2,1))$. Then $E$ is an effect algebra that is also an MV-algebra. It has a unique nontrivial ideal $I=\{(0,n): n \ge 0\}$, but $E/I\cong \Gamma(\frac{1}{2}\mathbb Z,1)$ and it has no isomorphic copy in $E$.

Now we show that if $\mathrm{LexId}(E)$ is nonempty, then $\mathrm{LexId}(E)$ is a linearly ordered set with respect to the set theoretical inclusion.

\begin{proposition}\label{pr:3.2}
Let $E$ be an effect algebra with \RDP. If $I,J\in \mbox{\rm LexId}(E)$, then $I\subseteq J$ or $J\subseteq I$. In addition, every lexicographic ideal is contained in $\Infinit(E)$. If one of the lexicographic ideals is a maximal ideal, then $E$ is local.
\end{proposition}

\begin{proof}
Suppose the converse, that is, there are $x \in I\setminus J$ and $y \in J\setminus I$. Then $x/I<y/I$ and $y/J<x/J$ which yields $x<y$ and $y<x$ which is absurd.

Assume that $I$ is a lexicographic ideal of the effect algebra $E$ and choose two elements $x,y \in I$. Then $x/I=0/I <1/I=y^-/I$ which entails $x<y^-$ since $I$ is strict. Hence, $x+y$ is defined in $E$ and $x+y\in I$. Consequently, every element of $I$ is infinitesimal, and $I \subseteq \Infinit(E)$.

Let $I$ be a lexicographic ideal that is also a maximal ideal of $E$. Assume that $J$ is another maximal ideal of $E$, and let $J\ne I$. There are $x \in I\setminus J $ and $y \in J\setminus I$ which implies $x <y$ so that $x \in J$ which is a contradiction. Hence, $I$ is a unique maximal ideal of $E$.
\end{proof}

It can happen that $\mathrm{StrId}(E)=\emptyset$. Let $E=[0,1]\times [0,1]=\Gamma(\mathbb R^2,(1,1))$. It satisfies the RDP, and there are only two nontrivial ideals $I_1=[0,1]\times \{0\}$ and $I_2=\{0\}\times [0,1]$. Then $E/I_1\cong \Gamma(\mathbb R,1)\cong E/I_2$ and it is straightforward to see that there is no strict nontrivial ideal of $E$.

\begin{proposition}\label{pr:3.3}
Let $E$ be an effect algebra with \RDP. If $\mathrm{StrId}(E)\ne \emptyset$, the set of strict ideals is linearly ordered with respect to set theoretical inclusion and there is the largest strict ideal of $\mathrm{StrId}(E)$. Every strict ideal of $E$ is contained in $\Infinit(E)$.
\end{proposition}

\begin{proof}
The fact that $\mathrm{StrId}(E)$ is linearly ordered with respect to the set theoretical inclusion can be proved in the same way as that in Proposition \ref{pr:3.2}

Let $\{I_\alpha\}_\alpha$ be the system of all strict ideals of $E$, it is a chain with respect to the set theoretical inclusion. Set $I=\bigcup_\alpha I_\alpha$. Clearly that $I$ is an ideal of $E$ and $1 \notin I$.  Now we show that $I$ is strict. Let $x \in I$ and $y\notin I$. There is $\alpha$ such that $x \in I_\alpha$ and hence, $y \notin I_\alpha$ which gives $x< y$.

Let $I$ be a strict ideal of $E$ and let $x,y \in I$ be given. Then from $x/I<y^-/I$ we get $x<y^-$ so that $x+y$ exists in $E$ and $x+y$ belongs to $I$. Hence, every element of $I$ is an infinitesimal, so that $I \subseteq \Infinit(E)$.
\end{proof}

\begin{corollary}\label{co:3.4}
Let $E$ be an effect algebra with \RDP. Then every state $s$ on $E$ vanishes on each lexicographic ideal and on each strict ideal of $E$.
\end{corollary}

\begin{proof}
Since $E$ satisfies the RDP, $E$ admits a state. Let $s$ be an arbitrary state on $E$. By Propositions \ref{pr:3.2}--\ref{pr:3.3}, every element of a lexicographic ideal and every element of a strict ideal is infinitesimal. Hence, $s$ vanishes on such ideals.
\end{proof}

Other results on states on the lexicographic effect algebras will be presented in Theorems \ref{th:4.4}--\ref{th:4.5} below.

\begin{proposition}\label{pr:3.5}
Let $(H,u)$ be an antilattice unital Abelian po-group  with the \RDP and $G$ an Abelian directed po-group with the \RDP. If we set $I=\{(0,g): g \in G^+\}$, then $I$ is a lexicographic ideal of $E=\Gamma(H\lex G,(u,0))$, whenever $H$ and $G$ are nontrivial po-groups.
\end{proposition}

\begin{proof}
The set $I:=E_0 =\{0\}\times G^+$ is an ideal that is evidently retractive. It is prime while $E/E_0\cong \Gamma(H,u)$ and $(H,u)$ is an antilattice. Let $s:E\to E/E_0\cong \Gamma(H,u)$. Then $s$ is an $(H,u)$-valued state and $E_t=s^{-1}(\{t\})$ defines an ordered and directed $(H,u)$-decomposition of $E$.  Now let $x/E_0 < y/E_0$. Then $t:=s(x)< v:=s(y)$ proving that $x\in E_t$ and $y \in E_v$ so that $x<y$. Hence, $E_0$ is strict. Consequently, $I=E_0$ is lexicographic.
\end{proof}

\section{$(H,u)$-perfect Effect Algebras}

The following notion was already studied for pseudo MV-algebras in \cite{275} and as a special kind of pseudo effect algebras in \cite{DvKr, DvKo}.

In what follows, we will assume that $(H,u)$ is a fixed nontrivial Abelian unital po-group.

For a unital Abelian po-group $(H,u)$, we set $[0,u]_H:=\{h\in H: 0\le h\le u\}.$

We say that a decomposition $(E_t: t \in [0,u]_H)$ of an effect algebra $E,$ i.e. a system $(E_t: t \in [0,u]_H)$ of nonempty subsets of $E$ such that  $E_s \cap E_t = \emptyset$ for $s\ne t,$ $s,t \in [0,u]_H$ and $\bigcup_{t \in [0,u]_H} E_t = E$, is
an  $(H,u)$-{\it decomposition} if

\begin{enumerate}
\item[{\rm (a)}] $E_{t}^{-}=E_{u-t}$ for any $t\in
[0,u]_H$;

\item[{\rm (b)}] if $x\in E_{s},$ $y\in E_{t}$ and $x+ y$ exists in
$E$, then $s+t\le u$ and $x+ y\in E_{s+t}$ for
$s,t \in [0,u]_H$.
\end{enumerate}

For example, if $E=\Gamma(H\lex G,(u,0))$, where $(H,u)$ is a unital Abelian po-group and $G$ is a directed po-group with RDP, then $(E_t: t \in [0,u]_H)$, where $E_t=\{(t,g): (t,g)\in E\}$, is an $(H,u)$-decomposition of $E$. In addition, $E_0$ is an ideal of $E$, and $E_0\subseteq \Infinit(E)$.

We say that a mapping $s:E \to [0,u]_H$ such that (i) $s(1)=u$, and (ii) $s(x+y) = s(x)+s(y)$ whenever $x+y$ is defined in $E$ is an $(H,u)$-{\it state}. If in addition (iii) $s(E)=[0,u]_H$, $s$ is said to be an $(H,u)$-{\it valued state}.

If $(H,u)$ is a unital Abelian po-group with RDP, and $E$ is an effect algebra with the RDP, then every $(H,u)$-state $s:E\to [0,u]_H$ can be uniquely extended to a po-group homomorphism $\hat s:K\to H$ such that $\hat s(k)=u$, where $(K,k)$  is an Abelian unital po-group with RDP such that $E=\Gamma(K,k)$. If, in addition,  $s$ is an $(H,u)$-valued state, then the po-group homomorphism is surjective.

The properties of an $(H,u)$-state: (i) $s(0)=0$, (ii) $s(x)\le s(y)$ whenever $x\le y$, (iii) $s(x^-)=s(x)^-$, $x \in E$.

We note that it can happen, that for some effect algebras, there is no $(\mathbb R,1)$-state.

\begin{theorem}\label{th:4.1}
Let  $E$ be an effect algebra and $(H,u)$ be an Abelian unital po-group. The following two statements
are equivalent:
\begin{enumerate}
\item[{\rm (i)}] There exists an $(H,u)$-valued state on $E$.

\item[{\rm (ii)}] There exists an $(H,u)$-decomposition $(E_t: t \in [0,u]_H)$ of nonempty  subsets of $E.$
\end{enumerate}

In addition, there is a one-to-one correspondence between the set of all $(H,u)$-decompositions and the set of all $(H,u)$-valued states on $E$.
\end{theorem}

\begin{proof}
Let $s$ be an $(H,u)$-valued state. For any $t \in [0,u]_H$, we define $E_t:= s^{-1}(\{t\}).$ We assert that  the system $(E_t: t \in [0,u]_H)$ is a decomposition of $E.$ Indeed, for (a), let $x \in E_t.$ Then $s(x)=t$ and $s(x^-)=u-s(x)$, which yields $x^- \in E_{u-t}.$ Conversely, if $y \in E_{u-t},$ then $y^- \in E_t.$ For (b), assume $x \in E_s$ and $y \in E_t$ and let $x+y$ be defined in $E.$ Then $s(x+y) = s(x)+s(y)=s+t\le u,$ which implies $x+y \in E_{s+t}.$ Then $(E_t: t \in [0,u]_H)$ is an $(H,u)$-decomposition of $E.$

Conversely, let (ii) hold. Define a mapping $s: E \to [0,u]_H$ by $s(x)=t$ iff $x \in E_t.$ Choose two elements $x,y \in E$ with $x+y\in E$ such that $s(x)= t_1$ and $s(y)=t_2.$ Since $x \in E_{t_1}$ and $y \in E_{t_2},$ due to (b), we have $t_1+t_2 \le u.$ Hence, $s(x+y)=s(x)+s(y).$ There exists a unique $t \in [0,u]_H$ such that $0\in E_t.$ For every $x \in E_u,$ $x+0=x$, thus by (b), $t+u \le u$ which yields $t = 0$ and therefore, $s(0)=0$ and $s(1)=u.$ In other words, $s$ is an $(H,u)$-valued state.

Finally, let $\mathcal D=(E_t: t \in [0,u]_H)$ be an $(H,u)$-decomposition on $E$. The mapping $f(\mathcal D) =s$, where $s(x)=t$ iff $x \in E_t$ for $t \in [0,u]_H$ defines by (i) and (ii) a one-to-one correspondence in question.
\end{proof}

We say that an $(H,u)$-decomposition $(E_t: t \in [0,u]_H)$ of $E$ is {\it ordered} if, for $s<t,$ $s,t \in [0,u]_H,$ we have $E_s \leqslant E_t.$

\begin{theorem}\label{th:4.2}
Let $(E_t: t \in [0,u]_ H)$ be an $(H,u)$-decomposition of an effect algebra $E$, where $(H,u)$ is an Abelian unital po-group. Then $(E_t: t \in [0,u]_H)$ is ordered if and only if $E_{w}+ E_{v}$ exists in $E$ whenever $w+v<u$ for  $w,v \in [0,u]_H.$

In any such case,
\begin{itemize}
\item[{\rm (i)}] $E_{0}+E_0=E_0 \subseteq \Infinit(E)$ and $E_0$ is a Riesz ideal of $E$. $E_0$ and $E_u$ are directed. In addition, if $\Gamma(H,u)$ is Archimedean, then $E_0=\Infinit(E)$.

\item[{\rm (ii)}] $E_{s}+ E_{v}=E_{s+v}$ whenever $s+v<u$.

\item[{\rm (iii)}] If $t+v>u$, for any $x\in E_{t}$ and $y\in E_{v},$ then $x+ y$  does not exist in $E$.
\end{itemize}

\end{theorem}

\begin{proof}
If $(E_t:t \in [0,u]_H)$ is an $(H,u)$-decomposition, by Theorem \ref{th:4.1}, there is a unique  $(H,u)$-valued state $s$ such that $s(E_t)=t$ for each $t \in [0,u]_H.$

Assume $(E_t: t \in [0,u]_H)$ is ordered. Choose $t,v\in [0,u]_H$ with $t+v<u.$ We have that $t<u-v$, and  $E_{t}\leqslant E_{u-v}=E_{v}^{-},$ which implies that $E_{t}+
E_v$ exists and we show $E_{t}+ E_{v}=E_{t+v}.$ Indeed, for any $a\in E_{t}$ and any $b\in E_{v},$  we have $s(a+ b)=t+v$,
which implies that $a+ b\in E_{t+v}.$ Conversely, let $c\in E_{t+v}.$ For any $a\in E_{t}$, we have that $a\le c.$ Then there exists an element  $b\in E$ such that $a+ b=c.$ Hence, $s(a+ b)=s(a)+s(b)=t+v,$ then $s(b)=v$, which implies that $b\in E_{v}.$

Now let $E_t + E_w$ exist in $E$ for $t+w<u,$ $t,w \in [0,u]_H.$ Choose $y <v,$ $y,v \in [0,u]_H.$ Then $y +(u-v)<u$ so that $E_y + E_{u-v}=E_y+E_v^-$ exists in $E$ which yields $E_y \leqslant E_v.$

(i) For any $x, y\in E_{0},$ we have that $x+ y$ exists in $E$ because $x\le y^-\in E_0^-$. In addition, we have $s(x+y)=s(x)+s(y)=0$ and $x+y \in E_0$ which implies
$E_{0}\subseteq \Infinit(E).$ Now if $x\in E_0$ and, for $z\in E$, we have $z\le x$, then $0\le s(z)\le s(x)=0$ proving $E_0$ is an ideal of $E$.

Now let $i\in E_0$ and $i \le a+b$ for some $a,b \in E$. There are three cases (1) $a,b \in E_0$, then the statement is evident. (2) Only one from $a,b$ is in $E_0$. Without loss of generality, we can assume $a\in E_0$ and $b\notin E_0$. Then $i\le b$ and $i \le a+i \le a+b$. Finally, the third case (3) $a,b \notin E_0$. Then $i\le a,b$ and $i \le i+i \le a+b$.

Let $x,y \in E_0$, then $0\le x,y$ and $x,y \le x+y$ which yields $E_0$ is directed. Since $E_u=E_0^-$, we have that also $E_u$ is directed.

Now let $\Gamma(H,u)$ be Archimedean. For any  $x\in \Infinit(E),$ we have that $mx$  is defined in $E$ for each integer $m\ge 1.$ Then $s(mx)=ms(x) \leqslant u$ which implies $s(x)=0$ and $x\in E_{0}.$

(ii) Assume that $a\in E_{s}$ and $b\in E_{v}$ for $s+v<u.$ Then
$a+ b$ exists and $s(a+ b)= s+v,$ and so
$a+ b\in E_{s+v}.$ Conversely, let $z\in E_{s+v},$ then for any
$x\in E_{s}$, $x< z$, so that $z=x+(z-x),$ by
$s(z)=s(x)+s(z-x),$ which implies that $z-x\in E_{v}.$

(iii) Assume that $t+v>u$, $x\in E_{t},$ $y\in E_{v},$ and
$x+ y$ exists in $E$. Then we have  $s(x+ y)= s(x)+s(y)=t+v>u$  which is absurd.
\end{proof}

We note that nevertheless $E_0$ in Theorem \ref{th:4.2} is an ideal, it is not necessary maximal. Indeed, take the effect algebra from Example \ref{ex:3.1}, then the ideal $I_1$ is such a case of an ordered $(H,u)$-decomposition. A characterization when $E_0$ is a maximal ideal will be shown in Proposition \ref{pr:6.2} below.

By (i) of Theorem \ref{th:4.2}, we see that $E_0$ and $E_u$ are always directed when the $(H,u)$-decomposition is ordered. Motivating by this, we say that an $(H,u)$-decomposition $(H_t: t \in [0,u]_H)$ of an effect algebra $E$ is {\it directed} if any $E_t$, $t \in [0,u]_H$, is directed. For example, (1) if $G$ is a directed Abelian po-group and $(H,u)$, unital Abelian po-group, and $E=\Gamma(H \lex G,(u,0))$, then every $E_t=\{(t,g): (t,g)\in E\}$ is directed. (2) If $(H,u)=(\frac{1}{n}\mathbb R,1)$, $E$ satisfies RDP, then any ordered $(\frac{1}{n}\mathbb R,1)$-decomposition of $E$ is directed, \cite[Prop 5.12]{DXY}.

\begin{theorem}\label{th:4.3}
Let $(E_t: t \in [0,u]_ H)$ be an ordered and directed $(H,u)$-decomposition of an effect algebra $E$, where $(H,u)$ is an Abelian unital po-group. Then
\begin{itemize}
\item[{\rm (i)}] $E/E_0\cong \Gamma(H,u)$.

\item[{\rm (ii)}]
If $(E_t: t \in [0,u])$ and $(E_t': t \in [0,u])$ are two ordered and directed $(H,u)$-decompositions of $E$, then $E_t=E_t'$ for each $t \in [0,u]_H.$

\end{itemize}
\end{theorem}

\begin{proof}
(i) By Theorem \ref{th:4.2}(i), $E_0$ is a Riesz ideal of $E$ and by \cite[Cor 3.1.17]{DvPu}, $E/E_0$ is an effect algebra. We have $x\sim_{E_0} y$ iff $x-a=y-b$ for some $a,b \in E_0$ and $a\le x$, $y\le b$. This is entails $x,y \in E_h$ for some $h \in [0,u]_H$. Conversely, let $x,y \in E_h$ for some $h\in [0,u]_H$. Since $E_h$ is directed, there is an element $z\in E_h$ such that $z\le z,y$. Then both $a=x-z$, $b=y-z$ belong to $E_0$ and $x-a=x-(x-z)=z=y-b$, so that $x\sim_{E_0} y$. We define a mapping $\phi: E/E_0\to \Gamma(H,u)$ by $\phi(x/E_0)=h$ iff $x \in E_h$. The mapping $\phi$ is an isomorphism in question. In addition, $\{x\in E: \phi(x/E_0)=h\}=E_h$ for every $h\in [0,u]_H$.

(ii) We assert  $E_0=E_0'$. Indeed, let $\pi_{E_0}: E\to E/E_0$ and $\pi_{E_0'}:E\to E/E_0'$ denote the canonical projections. By (i), $E/E_0\cong \Gamma(H,u)\cong E/E_0'$, so that there is an isomorphism $\phi: E/E_0\to E/E_0'$ with $\phi\circ \pi_{E_0} = \pi_{E_0'}.$ Hence, $E_0=\pi_{E_0}^{-1}(\{0\})=\pi_{E_0}^{-1}(\phi^{-1}(\{0\}))= \pi_{E_0'}^{-1}(\{0\})=E_0'$.

Now if $x\sim_{E_0} y$, then $x,y \in E_h$ for some $h \in [0,u]_H$, as well as $x\sim_{E_0'} y$ implies $x,y \in E'_{h'}$ and $h=h'$ which implies $E_h = E_h'$ for any $h \in [0,u]_H$.
\end{proof}

Since $\Gamma(H\lex G,(u,0))$ is an interval effect algebra, it has at least one state. In the next theorem, we show a criterion when it has only one state.

\begin{theorem}\label{th:4.4}
Let $E =\Gamma(H\lex G,(u,0))$, where $(H,u)$ is an Abelian unital po-group with \RDP\, and $G$ is a directed po-group with \RDP. Then $E$ has a unique state if and only if
$\Gamma(H,u)$ has a unique state $s_H$. In such a case, if $s_H$ is a unique state on $\Gamma(H,u)$,
then $E$ has also a unique state $s_0$, namely $s_0(h,g)=s_H(h,g)$, $(h,g)\in E$.
\end{theorem}

\begin{proof}
Let $\Gamma(H,u)$ have a unique state.
According to Theorem \ref{th:2.1} and \cite[Thm 4.19]{Goo}, $\Gamma(H,u)$ has a unique state iff there exist positive integers $a\ge b$ such that, given $h \in [0,u]_H$, there is some positive integer $n$ for which either $nbh\le na(u-h)$ or $nb(u-h)\le nah$ (calculated in the po-group $H$). If $h=0$, we have $nb(0,g)<n(a+1)(u,-g)$ $(g \in G^+)$ and if $h=u$, then $nb(u-h,-g)< n(a+1)(h,g)$ $(g \in G^+)$. Now let $0<h<u$.
Having these $a,b,n$ we have $nbh<n(a+1)(u-h)$ or $nb(u-h)<n(a+1)h$. Therefore, $nb(h,g)<n(a+1)(u-h,-g)$ or $nb(u-h,-g)<n(a+1)(h,g)$ $(g \in G)$. Therefore, if $\Gamma(H,u)$ has a unique state $s_H$, then $E$ has also a unique state, $s_0$ say. Then $s_0(h,g):=s_F(h)$, $(h,g)\in E$. Indeed, if $(E_h: h \in [0,u]_H)$ is an $(H,u)$-decomposition of $E$, by Theorem \ref{th:4.1}, there is an $(H,u)$-valued state on $E$. Then $s_H\circ s$ is a unique state on $E$.

Conversely, let $E$ have a unique state. Using the criterion \cite[Thm 4.19]{Goo}, there exist positive integers $a\ge b$ such that, given $(h,g)\in E$, there is an integer $n$ for which either $nb(h,g)\le na(u-h,-g)$ or $nb(u-h,-g)\le na(h,g)$. Then for these $a,b,n$ we have $nbh\le na(u-h)$ or $nb(u-h)\le nah$, proving $\Gamma(H,u)$ has a unique state.
\end{proof}

For example, if $(H,u)$ is a linearly ordered po-group with RDP, then $\Gamma(H,u)$ has a unique state. Consequently, $E=\Gamma(H\lex G,(u,0))$ has also a unique state however, $E$ is not necessarily linearly, e.g. if $G$ is not linearly ordered.

Theorem \ref{th:4.4} can be proved also using different arguments:

\begin{theorem}\label{th:4.5}
Let $E =\Gamma(H\lex G,(u,0))$, where $(H,u)$ is an Abelian unital po-group with \RDP\, and $G$ is a directed po-group with \RDP. Then there is a one-to-one correspondence between states on $E$ and on $\Gamma(H,u)$, respectively.
\end{theorem}

\begin{proof}
Let $s$ be a state on $E$. Then $s(h,g)=s(h,0)$ for all $(h,g)\in E$. Indeed, let $\hat s$ be the unique extension of $s$ to $(H\lex G,(u,0))$. Since $(0,g) \in E$ for any $g\ge 0$, we have $0\le s(n(0,g))=ns(0,g)\le 1$, so that $s(0,g)=0$. Consequently, $\hat s(0,-g)=0$ for any $g \ge 0$ and finally, $\hat s(0,g)=0$ for any $g \in G$. In addition, two states $s_1$ and $s_2$ on $E$ coincide iff $s_1(h,0)=s_2(h,0)$ for any $h \in \Gamma(H,u)$. Therefore, the state $s$ on $E$ induces a unique state $s_H$ on $\Gamma(H,u)$ such that $s_H(h)=s(h,0)$.

Conversely, let $s_H$ be a state on $\Gamma(H,u)$. Let $(E_h: h \in [0,u]_H)$ be the $(H,u)$-decomposition corresponding to $(H\lex G,(u,0))$. By Theorem \ref{th:4.1}, there is a unique $(H,u)$-valued measure $\sigma$ such that $\sigma^{-1}(\{h\})= E_h$. Then $s:=s_H \circ \sigma$ is a state on $E$. If $s_H'$ is any state on $\Gamma(H,u)$, then $s_H$ and $s'_H$ are different iff $s_H(h)\ne s'_H(h)$ for some $h \in [0,u]_H$. Then $s_H\circ \sigma$ and $s'_H\circ \sigma$ are different states on $E$. Hence, there is a one-to-one correspondence between the set of states on $E$ and $\Gamma(H,u)$.
\end{proof}

\section{Representation of Strong $(H,u)$-perfect Effect Algebras}

In this section, we give algebraic conditions that guarantee that an effect algebra is lexicographic with a nontrivial Abelian unital po-group $(H,u)$ with RDP.

We say that an effect algebra $E$ is $(H,u)$-{\it perfect} if there is an ordered and directed $(H,u)$-decomposition $(H_t: t \in [0,u]_H)$, where $(H,u)$ is a nontrivial Abelian unital po-group.

For example,  (i) if $(H,u)=(\mathbb Z,1)$ and $E$ is an MV-algebra, we are speaking on  a {\it perfect} MV-algebra, \cite{DiLe1}, (ii)  if $(\mathbb H,u)=(\frac{1}{n}\mathbb Z,1),$ a $(\frac{1}{n}\mathbb Z,1)$-perfect pseudo MV-algebra is said to be $n$-{\it perfect}, see \cite{Dv08}, (iii) if $\mathbb H$ is a subgroup of the group of real numbers $\mathbb R$, such that $1 \in \mathbb H$, $(\mathbb H,1)$-perfect pseudo MV-algebras are in \cite{264} called $\mathbb H$-{\it perfect} pseudo MV-algebras.

In what follows, we define a stronger version of $(H,u)$-perfect effect algebras, called strong $(H,u)$-perfect effect algebras, see \cite{Dv08, 264, DXY}. We say that a directed and ordered $(H,u)$-decomposition $(E_t: t \in [0,u]_H)$ is a {\it strong} $(H,u)$-{\it decomposition}, if there is a system $(c_t\in M: t \in [0,u]_ H)$ of elements of $E$ such that

\begin{enumerate}
\item[{\rm (i)}] $c_t \in M_t$ for each $t \in [0,u]_ H$;
\item[{\rm (ii)}] if $v+t \le u,$ $v,t \in [0,u]_ H,$ then $c_v+c_t=c_{v+t}$;
\item[{\rm (iii)}] $c_u=1.$
\end{enumerate}
Then $c_0+c_0=c_0$, so that $c_0=0.$

We say that an effect algebra $E$ with RDP is {\it strong} $(H,u)$-{\it perfect} if there is a strong $(H,u)$-decomposition $(E_t: t \in [0,u]_H)$ of $E$.

For example, if $G$ is a directed Abelian po-group with RDP and $(H,u)$ is an Abelian unital po-group with RDP, then $E=\Gamma(G\lex H,(u,0))$ is a strong $(H,u)$-perfect effect algebra, indeed, the system $(c_t: t \in [0,u]_H)$, where $c_t=(t,0)$, satisfy above conditions (i)--(iii).

Let $E$ be an effect algebra with RDP such that $(E_0,E_1)$ is a $(\mathbb Z,1)$-perfect decomposition of $E$. By Theorem \ref{th:4.2}(i), $E_0$ and $E_1$ are directed. If in addition, $E_0\leqslant E_1$, then $E$ is a strong $(\mathbb Z,1)$-perfect effect algebra because $c_0=0\in E_0$, $c_1=1\in E_1$ are necessary elements establishing that $(E_0,E_1)$ is a strong $(\mathbb Z,1)$-perfect decomposition. We call those effect algebras simply {\it perfect}, they are equivalent to those studied in \cite{177}.

\begin{theorem}\label{th:5.1}
Let $E$ be an effect algebra with \RDP and let $(H,u)$ be an Abelian unital po-group with \RDP. Then $E$ is strong $(H,u)$-perfect if and only if there exists an Abelian directed po-group $G$ such that $E \cong \Gamma(H\lex G,(u,0))$.

If this is a case, $G$ is unique and it satisfies \RDP.
\end{theorem}

\begin{proof}
If there is a directed Abelian po-group $G$ such that $E$ is isomorphic to $\Gamma(H\lex G,(u,0))$, then $E$  has RDP and is a strong $(H,u)$-perfect effect algebra with RDP.

Conversely, assume that $E$ is a strong $(H,u)$-perfect effect algebra. There is a directed and ordered strong $(H,u)$-perfect decomposition $(H_t: t \in [0,u]_H)$ of $E$ and there is a directed unital Abelian po-group $(K,v)$ with RDP such that $E=\Gamma(K,v)$.

By (i) of Theorem \ref{th:4.2}, $E_0$ is an associative cancellative semigroup satisfying conditions of Birkhoff's Theorem  \cite[Thm XIV.2.1]{Bir}, \cite[Thm II.4]{Fuc},
which guarantee that $E_0$ is a positive cone of
a unique (up to isomorphism) Abelian directed po-group $G$. In addition, $G$ satisfies RDP while $E_0$ does.

Take the $(H,u)$-strong perfect effect algebra $\mathcal E_{H,u}(G)$ defined by
$$\mathcal E_{H,u}(G):=\Gamma(H\lex G,(u,0)),\eqno(5.1)
$$
and define a mapping $\phi: E \to \mathcal E_{H,u}(G)$ by
$$
\phi(x):= (t, x - c_t)\eqno (5.2)
$$
whenever $x \in E_t$ for some $t \in [0,u]_H,$ where $ x-c_t$ denotes the difference taken in the group $K$.

\vspace{2mm}
{\it Claim 1:} {\it  $\phi$ is a well-defined mapping.}
 \vspace{2mm}

Indeed, $E_0$  is in fact the positive cone of a directed po-group $G$ which is a po-subgroup of $K.$  Let $x \in E_t$ and $c_t \in E_t$. If $t=0$, then $\phi(x)= (0,x)$ and since $x \in E_0=G^+$, $\phi(x)$ is correctly defined. If $t=u$, then $\phi(x)=(u,x-1)$ and $x-1\in G^-$, so again $\phi(x)$ correct. Now let $0<t<u$.
Since every $E_t$ is directed, there is an element $a\in E_t$ such that $a\le x,c_t$. Then $x-c_t = x- a+a-c_t= (x-a)-(c_t-a)\in G$ because $x-a,c_t-a\in E_0=G^+$.

\vspace{2mm}
{\it Claim 2:} {\it The mapping $\phi$ is an injective and surjective homomorphism of effect algebras.}

\vspace{2mm}

We have $\phi(0)=(0,0)$ and $\phi(1)=(u,0).$ Let $x \in E_t.$ Then $x^- \in M_{u-t},$ and $\phi(x^-) =(u-t, x -  c_{u-t}) = (u,0)-(t,x - c_t)=\phi(x)^-.$
Now let $x,y \in E$ and let $x+y$ be defined in $E$. Then $x\in E_{t_1}$ and $y \in E_{t_2}$ and $t_1+t_2\le u$. Since $x\le y^-,$ we have $t_1 \le u-t_2$, so that $\phi(x) \le \phi(y^-)=\phi(y)^-$ which means  $\phi(x)+\phi(y)$ is defined in $\mathcal E_{H,u} (G).$ Then $\phi(x+y) = (t_1+t_2, x+y - c_{t_1+t_2}) =
(t_1+t_2, x+y -(c_{t_1} + c_{t_2}))= (t_1,x-c_{t_1}) + (t_2,y- c_{t_2})=\phi(x)+\phi(y).$ Hence, $\phi$ is a homomorphism of effect algebras.

Assume $\phi(x)=\phi(y)$ for some $x\in E_{t}$ and $y \in E_v.$ Then $(t,x-c_t)= (v, y - c_v)$ which yields $t=v$, $c_t=c_v$ and consequently, $x=y$, so that $\phi$ is injective.

To prove that $\phi$ is surjective, assume two cases: (i) Take $g \in G^+=E_0.$  Then $\phi(g)=(0,g).$ In addition $g^- \in M_u$ so that $\phi(g^-) = \phi(g)^-= (0,g)^- = (u,0)-(0,g)=(u,-g).$ (ii) Let $g \in G$ and $t$ with $0<t<u$ be given. Then $g = g_1-g_2,$ where $g_1,g_2 \in G^+=E_0.$ Since $c_t \in E_t,$ $g_1 + c_t$ exists in $E$ and it belongs to $E_t,$ and $g_2 \le g_1+c_t$ which yields $(g_1+c_t)- g_2 \in E_t.$  Hence, $g+c_t = (g_1 + c_t)- g_2 \in E_t$ which entails $\phi(g+c_t)=(t,g).$

Consequently, $E$ is isomorphic to the effect algebra $\mathcal E_{H,u}(G)$.

If $E \cong \Gamma(H\lex G',(u,0))$ for some  directed unital Abelian po-group $G'$ with RDP, then $(H\lex G,(u,0))$ and $(H\lex G',(u,0))$ are isomorphic unital po-groups in view of the categorical equivalence, see \cite[Thm 5.8]{177} or Theorem \ref{th:2.1}; let $\psi: \Gamma(H\lex G,(u,0)) \to \Gamma(H\lex G',(u,0))$ be an isomorphism of the lexicographic products. Hence, by Theorem \ref{th:4.3}(ii), we see that $\psi(\{(0,g): g \in G^+\})= \{(0,g'): g' \in G'^+\}$ which proves that $G$ and $G'$ are isomorphic po-groups.
\end{proof}

We note that in Theorem \ref{th:5.1} the case $G=O$, the zero po-group, is not excluded. It can happen if $E_0$ is the zero ideal of $E$.

In the rest of the section we will assume that $(H,u)$ is a fixed Abelian unital po-group with RDP. Theorem \ref{th:5.1} enables us to prove the following categorical equivalence of the category of strong $(H,u)$-perfect effect algebras with the category of directed Abelian po-groups with RDP.

Let $\mathcal {SPEA}_{H,u}$ be the category of strong $(H,u)$-perfect  effect algebras whose objects are strong $(H,u)$-perfect effect algebras and morphisms are homomorphisms of effect algebras. Now let $\mathcal{AG}$ be the category whose objects are directed Abelian po-groups with RDP  and morphisms are homomorphisms of po-groups.

Define a mapping $\mathcal E_{H,u}: \mathcal G \to  \mathcal {SPEA}_{H,u}$ as follows: for $G\in \mathcal{AG},$ let
$$
\mathcal E_{H,u}(G):= \Gamma(H\lex G,(u,0))
$$
and if $h: G \to G_1$ is a po-group homomorphism, then

$$
\mathcal E_{H,u}(h)(t,g)= (t, h(g)), \quad (t,g) \in \Gamma(H\lex G,(u,0)).
$$
It is evident that $\mathcal E_{H,u}$ is a functor.

\begin{proposition}\label{pr:5.2}
$\mathcal E_{H,u}$ is a faithful and full
functor from the category $\mathcal{AG}$ of directed po-groups with \RDP\,  into the
category $\mathcal{SPEA}_{H,u}$ of strong $(H,u)$-perfect effect algebras.
\end{proposition}

\begin{proof}
Let $h_1$ and $h_2$ be two morphisms from $G$
into $G'$, $G,G'\in \mathcal{AG}$, such that $\mathcal E_{H,u}(h_1) = \mathcal E_{H,u}(h_2)$. Then
$(0,h_1(g)) = (0,h_2(g))$ for each $g \in G^+$, consequently $h_1 =
h_2.$

To prove that $\mathcal E_{H,u}$ is a full  functor, suppose that

Now let $f: \Gamma(H\lex G, (u,0)) \to \Gamma(H\lex G', (u,0))$ be a morphism.  Then $f(0,g)
= (0,g')$ for a unique $g' \in G'^+$. Define a mapping $h:\ G^+ \to
G'^+$ by $h(g) = g'$ iff $f(0,g) =(0,g').$ Then $h(g_1+g_2) = h(g_1)
+ h(g_2)$ if $g_1,g_2 \in G^+.$
Assume now that $g \in G$ is arbitrary.  If $g=g_1-g_2=g_1'-g_2'$ where $g_1,g_2,g_1',g_2 \in G^+$, then $g_1 +g_2'=g_1'+g_2$ and $h(g_1)+ h(g_2')=h(g_1')+h(g_2)$ which shows that $h(g) = h(g_1) - h(g_2)$ is a
well-defined extension of $h$ from $G^+$ onto $G$.

Let $0\le g_1 \le g_2.$ Then $(0,g_1)\le (0,g_2),$
which means  $h$ is a mapping preserving the partial order.

Finally, we have proved that $h$ is a homomorphism of po-groups, and $\mathcal E_{H,u}(h) = f$ as claimed.
\end{proof}

We note that by a {\it universal group}  for an
effect algebra $E$ with RDP we mean a pair $(G,\gamma)$ consisting of a directed Abelian po-group $G$ with RDP and of a $G$-valued measure $\gamma :\, E\to G^+$
(i.e., $\gamma (a+b) = \gamma(a) + \gamma(b)$ whenever $a+b$ is
defined in $E$) such that the following conditions hold:  (i)
$\gamma(E)$ generates ${ G}$. (ii) If $K$ is a group and
$\phi:\, E\to K$ is a $K$-valued measure, then there is a group
homomorphism ${\phi}^*:{ G}\to K$ such that $\phi ={\phi}^*\circ
\gamma$.

Due to \cite[Thm 7.2]{DvVe2}, every  effect algebra with RDP admits a universal group,
which is unique up to isomorphism, and $\phi^*$ is unique. The
universal group for $E = \Gamma(G,u)$ is $(G,id)$, where $id$ is the
embedding of $E$ into $G$.

Let $\mathcal A$ and $\mathcal B$ be two categories and let $f:\mathcal A \to \mathcal B$ be a functor. Suppose that $g,h$ be two functors from $\mathcal B$ to $\mathcal A$ such that $g\circ f = id_\mathcal A$ and $f\circ h = id_\mathcal B,$ then $g$ is a {\it left-adjoint} of $f$ and $h$ is a {\it right-adjoint} of $f.$

\begin{proposition}\label{pr:5.3}
The functor $\mathcal  E_{H,u}$ from the
category $\mathcal{AG}$ into  $\mathcal{SPEA}_{H,u}$ has  a left-adjoint.
\end{proposition}

\begin{proof}
We show, for a strong $(H,u)$-perfect effect algebra $E$ with a strong $(H,u)$-decomposition $(E_t: t \in [0,u]_H)$ endowed with a family $(c_t: t \in [0,u]_H)$ satisfying conditions (i)--(iii), there is a universal arrow $(G,f)$, i.e., $G$ is an object in $\mathcal{AG}$ and $f$ is a homomorphism from the effect algebra
$E$ into ${\mathcal  E}_{H,u}(G)$ such that if $G'$ is an object from $\mathcal{AG}$ and $f'$ is a homomorphism from $E$ into ${\mathcal  E}_{H,u}(G')$, then
there exists a unique morphism $f^*:\, G \to G'$ such that ${\mathcal
E}_{H,u}(f^*)\circ f = f'$.

By Theorem \ref{th:5.1}, there is a unique (up to isomorphism of po-groups) Abelian directed  po-group $G$  such that $E \cong \Gamma(H \lex G,(u,0)).$ By \cite[Thm 7.2]{DvVe2}, $(H \lex G, \gamma)$ is a universal group for $E,$ where $\gamma: E \to  \Gamma(H \lex G, (u,0))$ is defined by $\gamma(a) = (t,a -c_t),$ if $a \in E_t.$
\end{proof}

Define a mapping ${\mathcal  P}_{H,u}: \mathcal  {SPEA}_{H,u}\to \mathcal{AG}$
via ${\mathcal  P}_{H,u}(E) := G$ whenever $(H\lex  G, f)$ is a
universal group for $E$. It is clear that if $f_0$ is a morphism
from the effect algebra $E$ into another one $F$ with RDP, then $f_0$ can be uniquely extended to a po-group homomorphism ${\mathcal  P}_{H,u} (f_0)$ from $G$ into $G_1$, where $(H
\lex G_1, f_1)$ is a universal group for the strong $(H,u)$-perfect
effect algebra $F$.

\begin{proposition}\label{pr:5.4}
The mapping ${\mathcal  P}_{H,u}$ is a functor from the
category $\mathcal {SPEA}_{H,u}$ into the category $\mathcal{AG}$ which is a left-adjoint of the functor ${\mathcal  E}_{H,u}.$
\end{proposition}

\begin{proof}
 It follows from the properties of the
universal group.
\end{proof}

Now we present  the basic result on a categorical equivalence of the
category of strong $(H,u)$-perfect effect
algebras and the category of $\mathcal{AG}$.

\begin{theorem}\label{th:5.5}
The functor ${\mathcal  E}_{H,u}$ defines a categorical
equivalence of the category $\mathcal{AG}$ and the
category $\mathcal {SPEA}_{H,u}$ of strong $(H,u)$-perfect effect algebras.

In addition, if $h:\ {\mathcal  E}_{H,u}\mathbb (G) \to {\mathcal  E}_{H,u}(G')$ is a
homomorphism of effect algebras, then there is a unique homomorphism
$f:\ G \to G'$ of po-groups such that $h = {\mathcal  E}_{H,u}(f)$, and
\begin{enumerate}
\item[{\rm (i)}] if $h$ is surjective, so is $f$;
 \item[{\rm (ii)}] if $h$ is  injective, so is $f$.
\end{enumerate}
\end{theorem}

\begin{proof}
According to \cite[Thm IV.4.1]{MaL}, it is
necessary to show that, for a strong $(H,u)$-perfect effect algebra $E$, there is an
object $G$ in ${\mathcal  G}$ such that ${\mathcal  E}_{H,u}(G)$ is isomorphic to $E$. To show that, we take a universal group $(H
\lex G, f)$. Then ${\mathcal  E}_{H,u}(G)$ and $E$ are isomorphic.
\end{proof}

An easy corollary of Theorem \ref{th:5.5} is the following result:

\begin{corollary}\label{co:5.6}
If $(H',u')$ is an arbitrary Abelian unital po-group with \RDP, then $\mathcal{SPEA}_{H',u'}$ and $\mathcal{SPEA}_{H,u}$ are categorically equivalent categories of effect algebras.
\end{corollary}

\section{Effect Algebras with Retractive and Lexicographic Ideals}

In the section we describe lexicographic effect algebras studying retractive and lexicographic ideals. We describe some representation theorems, finishing with the result that every lexicographic effect algebra with RDP is a subdirect product of antilattice lexicographic effect algebras with RDP.

We remind that an effect algebra $E$ is {\it simple}, if $E$ has only trivial ideals, i.e. $\{0\}$ and $E$. From (2.2) we conclude that every simple effect algebra with RDP is Archimedean and in view of \cite[Cor 6.8]{177}, $E$ is an antilattice. If $E$ has the RDP, then  due to \cite[Prop 6.7]{177}, an ideal $I$ of $E$ is maximal iff $E/I$ is a simple effect algebra.

We recall that an {\it o-ideal} of a po-group $G$ is any directed convex subgroup of $G$. An o-ideal $I$ of a po-group $G$ is said to be (i) {\it maximal} if it is a proper subset of $G$
and it is not contained in any proper o-ideal of $G$, (ii) {\it prime} if, for all o-ideals $P$ and
$Q$ of $G$ with $P \cap Q \subseteq I$, we have $P \subseteq I$ or $J \subseteq I$.

If $E =\Gamma(G,u)$ and $E$ satisfies RDP, by \cite[Thm 6.11]{177}, there is a one-to-one correspondence between ideals of $E$ and o-ideals of $(G,u)$: every ideal of $E$ can be extended to a unique o-ideal of $G$, and the restriction of any o-ideal to $\Gamma(G,u)$ gives an ideal of $E$. In addition, maximal (prime) ideals correspond to maximal (prime) o-ideals of $G$. Hence, $E$ is simple iff $G$ is {\it simple}, i.e., $G$ has only two o-ideals: $\{0\}$ and $G$.
Moreover, by \cite[Thm 6.17]{177}, $E$ is an antilattice iff $G$ is an antilattice.

\begin{example}\label{ex:6.1}
Let $G$ be the additive group $\mathbb R^2$ with the positive cone of all $(x,y)\in \mathbb R^2$  such that either $x = y = 0$ or $x > 0$ and $y > 0$. Then $u =(1,1)$ is a strong unit  for $G$. The effect algebra $E = \Gamma(G,u)$ is an antilattice simple effect algebra with the
\RDP, but $E$ is not a lattice.
\end{example}

We note that according to Kadison's Antilattice Theorem, \cite[Thm 58.4]{LuZa}, the po-group $\mathcal B(H)$ is an antilattice, in addition RDP fails in $\mathcal B(H)$.

The relation between the Archimedean property of $E=\Gamma(G,u)$ and of $G$ is not straightforward. Of course, if $G$ is Archimedean, so is $E$. But for the converse, we know only partial results. If $E$ is a lattice with RDP, then due to \cite[Thm 4.6]{DvVe4}, both notions are equivalent. They are equivalent also in the case when $E$ satisfies RDP and is {\it supremum-homogeneous}, i.e. for any $a,
b_i \in E$, $i \in I$, such that either $\bigvee_i (a +
b_i)$ exists, also
$\bigvee_i b_i$ exists, cf. \cite[Thm 4.4]{DvVe4}.

We note that a po-group (or an effect algebra) $G$ is {\it monotone $\sigma$-complete} provided that every ascending sequence $x_1\le x_2 \le\cdots $ in $G$ which is bounded above in $G$ has a suprememum in $G$. By \cite[Prop 16.9]{Goo}, if $(G,u)$ satisfies RDP, then $G$ is monotone $\sigma$-complete iff $\Gamma(G,u)$ is monotone $\sigma$-complete. In addition, $G$ is Archimedean.

The following result speaks about maximality of the ideal $E_0$ in any directed and ordered  $(H,u)$-decomposition $(E_t: t \in [0,u]_H)$ of $E$ with RDP.

\begin{proposition}\label{pr:6.2}
Let $E$ be an effect algebra with \RDP\, and let $(E_t: t \in [0,u]_H)$ be an ordered and directed $(H,u)$-decomposition of $E$, where $(H,u)$ is an Abelian unital po-group. Then $E_0$ is a maximal ideal if and only if $H$ is a simple antilattice Abelian po-group. In such a case, $\Gamma(H,u)$ is Archimedean.
\end{proposition}

\begin{proof}
Let $E_0$ be a maximal ideal of $E$. By \cite[Prop 6.7]{177}, $E/E_0$ is a simple effect algebra. In view of Theorem \ref{th:4.3}(i), $E/E_0\cong \Gamma(H,u)$ which entails $\Gamma(H,u)$ is simple, Archimedean and an antilattice.  Consequently, $H$ is a simple antilattice po-group.

Conversely, assume that $H$ is a simple antilattice po-group, so that $\Gamma(G,u)$ is a simple antilattice. Since $E/E_0\cong \Gamma(G,u)$, by \cite[Prop 6.7]{177}, $E_0$ is a maximal ideal of $E_0$.
\end{proof}

Now we present an analogous result as Theorem \ref{th:5.1} which concerns retractive and strict ideals of local effect algebras. We note that if $E$ is a local MV-algebra, then $\Rad(E)$ is strict, see \cite[Lem 1.10]{DiLe2}; this can happen when $E$ is a lattice effect algebra with RDP. For general effect algebras such a property is unknown.

\begin{theorem}\label{th:6.3}
Let $E$ be an effect algebra with \RDP. The following statements are equivalent:

\begin{enumerate}
\item[{\rm (i)}] $E$ is local and $\Rad(E)$ is retractive and strict.
\item[{\rm (ii)}] $E$ is strong $(H,u)$-perfect for some antilattice, simple Abelian unital po-group $(H,u)$ such that $\Gamma(H,u)$ is Archimedean.
\item[{\rm (iii)}] There exist an antilattice, simple Abelian unital po-group $(H,u)$ such that $\Gamma(H,u)$ is Archimedean  and a directed po-group $G$ with \RDP\, such that $E \cong \Gamma(H\lex G, (u,0))$.
\end{enumerate}
\end{theorem}

\begin{proof}
(i) $\Rightarrow$ (ii)  Let $I$ be a unique maximal ideal of $E$ and let $(K,v)$ be a (unique up to isomorphism) unital po-group with RDP guaranteed by Theorem \ref{th:2.1} such that $E \cong \Gamma(K,v)$; without loss of generality we can assume that $E=\Gamma(K,v)$.

Since  $I=\Rad(E)$ is a retractive ideal, $E/I$ is an effect algebra with RDP that is isomorphic to some $\Gamma(H,u)$, where $(H,u)$ is an Abelian unital po-group with RDP.  Consequently, $\Gamma(H,u)$ can be injectively embedded into $K$ and $H$ is isomorphic to a subgroup of $K$.

In addition, let $\langle I\rangle$ be a subalgebra of $E$ generated by $I$. By Lemma \ref{le:2.2}, $\langle I\rangle = I \cup I^-$, and $\langle I\rangle$ is a perfect effect algebra with RDP. By \cite[Prop 5.3]{177} or Theorem \ref{th:5.1}, there is a unique (up to isomorphism) directed Abelian po-group $G$ with RDP such that $\langle I\rangle \cong \Gamma(\mathbb Z\lex G,(1,0))$.

In what follows, we prove that $E \cong \Gamma(H\lex G,(u,0))$.

Let $\pi_I$ be the canonical homomorphism from $E$ onto $E/I$ and $\phi:E/I\to \Gamma(H,u)$ be an isomorphism. Then $s=\phi\circ \pi_I$ is an $(H,u)$-valued measure. If we define  $E_t=s^{-1}(\{t\})$, $t \in [0,u]_H$, by Theorem \ref{th:4.1}, $(E_t:t \in [0,u]_H)$ is an $(H,u)$-decomposition of $E$, and since $E_w+E_v=E_{v+w}$ whenever $w+v<u$, we see by Theorem \ref{th:4.2} that $(E_t:t \in [0,u]_H)$ is an ordered $(H,u)$-decomposition.
It is evident that $E_0=I$.
Now we show that every $E_t$ is downwards directed. Let $a,b \in E_t$. Then $s(a)=s(b)$ so that $a\sim_I b$ which yields that there are two elements $e,f \in I$ such that $e\le a$, $f\le b$ and $a-e=b-f$. Then $a-e \in E_t$ and $a-e\le a,b$ which gives that every $E_t$ is also directed.

Since $E_0=I$, $E_0$ is a maximal ideal of $E$ and by Proposition \ref{pr:6.2}, $(H,u)$ is an antilattice Abelian simple unital po-group with RDP.

Being $I=\Rad(E)$ retractive, there is a unique effect subalgebra $E'$ of $E$ such that $s(E')=s(E)$. For any $t\in [0,u]_H$, there is a unique element $c_t \in E'$ such that $s(c_t)=t$. Then the system $(c_t \colon t \in [0,u]_H)$ satisfies the following properties (i) $c_t \in E_t$ for each $t \in [0,u]_H$, (ii) $c_{v+t}=c_v+c_t$ whenever $v+t\le u$, (iii) $c_u=u$.

In other words, we have proved that $(E_t: t \in [0,u]_H)$ is a strong $(H,u)$-perfect decomposition of $E$,  and consequently, $E$ is strong $(H,u)$-perfect which proves (ii).

(ii) $\Rightarrow$ (iii) By Theorem \ref{th:5.1}, $E \cong \Gamma(H\lex G',(u,0))$ for some unique (up to isomorphism) directed po-group $G'$ with RDP. Moreover, let $\iota: E \to \Gamma(H\lex G',(u,0))$ be an isomorphism. Then $\iota(I)=\iota(E_0)=\{0\} \times G'^+$. Since the po-groups $\{0\}\times G'^+$ and $\{0\}\times G^+$ are isomorphic, we have $G'$ and $G$ are also isomorphic, which proves $E \cong \Gamma(H\lex G,(u,0))$ as was asserted.

(iii) $\Rightarrow$ (i). By Proposition \ref{pr:6.2}, $E_0=\{0\}\times G^+$ is a maximal ideal of $E$. If there is another maximal ideal $J$ of $E$, then there is an element $x \in J\setminus E_0$ which yields $E_0\subseteq J$, which is absurd, and $E$ is local. Then $E_0$ is retractive  for $\Gamma(H\lex G,(u,0))$ as it was shown just after the definition of a retractive ideal.
\end{proof}

We note that the class of local effect algebras with RDP and with retractive $\Rad(E)$ is strictly included in the class of local effect algebras with RDP. Indeed, the effect algebra $E =\Gamma(\mathbb Z \lex \mathbb Z,(2,1))$ is linearly ordered, and it has a unique maximal ideal $I=\{0\}\times \mathbb Z^+$ that is not retractive as it was already shown.

In the case that $E$ is an MV-algebra, Theorem \ref{th:6.3} has the following formulation:

\begin{theorem}\label{th:6.4}
Let $E$ be an MV-algebra. The following statements are equivalent:

\begin{enumerate}
\item[{\rm (i)}] $E$ is local and $\Rad(E)$ is retractive.
\item[{\rm (ii)}] $E$ is strong $(\mathbb H,1)$-perfect for some subgroup $\mathbb H$ of $\mathbb R$ with $1 \in \mathbb H$.
\item[{\rm (iii)}] There exist a subgroup $\mathbb H$ of $\mathbb R$ with $1 \in \mathbb H$ and an $\ell$-group $G$ such that $E \cong \Gamma(\mathbb H\lex G, (1,0))$.
\end{enumerate}
\end{theorem}

\begin{proof}
It follows from \cite[Thm 5.7]{275}, or using Theorem \ref{th:6.3}: An MV-algebra $E\cong \Gamma(H\lex G,(u,0))$ is lexicographic if and only if $H$ is linear and $G$ is an $\ell$-group. If $(E_t: t \in [0,u]_H)$ is an $(H,u)$-decomposition, then it is strong $(H,u)$-decomposition. In such a case $\Gamma(H,u)\cong E/E_0$ and in view of  Proposition \ref{pr:6.2}, $E_0$ is a maximal ideal iff $(H,u)$ is a linear, Archimedean and Abelian unital $\ell$-group. By H\"older's theorem, \cite[Thm XIII.12]{Bir} or \cite[Thm IV.1.1]{Fuc}, it is isomorphic to some $(\mathbb H,1)$, where $\mathbb H$ is a subgroup of $\mathbb R$ and $1 \in \mathbb H$.
\end{proof}

The following theorem shows a similar relation as Theorem \ref{th:6.4}, however, the next result does not deal with local effect algebras as it does in Theorem \ref{th:6.4}.

\begin{theorem}\label{th:6.5}
Let $E$ be an effect algebra with \RDP\, and let $I$ be a lexicographic ideal of $E$. Then there is an Abelian antilattice unital po-group $(H,u)$ with \RDP\, such that $E/I \cong \Gamma(H,u)$ and there is an Abelian nontrivial directed po-group $G$ with \RDP\, and with $\langle I\rangle \cong \Gamma(\mathbb Z \lex G,(1,0))$ such that $E \cong \Gamma(H\lex G,(u,0))$.

Conversely, if  $E \cong \Gamma(H\lex G,(u,0))$, where $(H,u)$ is an Abelian antilattice unital po-group with \RDP\, and $G$ is an Abelian nontrivial directed po-group $G$ with \RDP, then $E$ has a lexicographic ideal $I$ such that $E/I \cong \Gamma(H,u)$ and $\langle I\rangle \cong \Gamma(\mathbb Z\lex G,(1,0))$.
\end{theorem}

\begin{proof}
In the same way as in the proof of Theorem \ref{th:6.3}, we can assume that $E=\Gamma(K,v)$ for some unital po-group $(K,v)$. Since $I$ is lexicographic, then $I$ is prime, so that $E/I$ is an antilattice effect algebra with RDP.  There is an Abelian antilattice unital po-group $(H,u)$ with RDP such that $E/I\cong \Gamma(H,u)$.

Let $\pi_I:E\to E/I$ be the canonical projection. For any $t \in [0,u]_H$, we set $E_t:=\pi_I^{-1}(\{t)\}$. We assert that $(E_t: t \in [0,u]_H)$ is an $(H,u)$-decomposition of $E$. Indeed, since $\pi_I$ is surjective, every $E_t$ is nonempty, and $\pi_I$ is in fact an $(H,u)$-valued state. By Theorem \ref{th:4.1}, $(E_t: t \in [0,u]_H)$ is an $(H,u)$-decomposition of $E$. In addition, let $x \in E_s$ and $y \in E_t$ for $s<t,$ $s,t \in [0,u]_H$. Then $\pi_I(x)=s <t<\pi(y)$ and $x<y$ because $I$ is strict. Therefore, the decomposition is ordered.

It is clear that $I=E_0$. Now let $x,y \in E_t$, then $\pi_I(x)=\pi_I(y)$ which means $x\sim_Iy$. There are $e,f \in E_0$ such that $e\le x$, $f \le y$, and $x-e=y-f$. Then $y-f \le x,y$ which proves that the decomposition is also downwards directed, and finally, directed.

Moreover, $\langle I\rangle = I \cup I^-$, and $\langle I\rangle$ is a perfect  effect algebra with RDP. By \cite[Prop 5.3]{177} or Theorem \ref{th:5.1}, there is a unique (up to isomorphism) directed Abelian po-group $G$ with RDP such that $\langle I\rangle \cong \Gamma(\mathbb Z\lex G,(1,0))$.

Now we show that $(E_t: t \in [0,u]_H)$ is a strong $(H,u)$-perfect decomposition of $E$. Being $I$ also retractive, there is a subalgebra $E'$ of $E$ such that $E'\cong E/I$ and $\pi_I(E')=\pi_I(E)$. In the same way as in the proof of implication (i) $\Rightarrow$ (ii) in Theorem \ref{th:6.3}, we find a family of elements $(c_t: t\in [0,u]_H)$ which proves that the decomposition is a strong $(H,u)$-perfect one. By Theorem \ref{th:5.1}, there is a directed Abelian po-group $G'$ with RDP such that $E \cong \Gamma(H\lex G',(u,0))$.

In the same way as in the proof of Theorem \ref{th:6.3}, we can show that $G\cong G'$ and finally, $E \cong \Gamma(H\lex G,(u,0))$.

The converse statement follows from Proposition \ref{pr:3.5} if we set $I=E_0=\{0\}\times G^+$.
\end{proof}

In the definition of a lexicographic ideal it was assumed that  it is a nonzero ideal. If we define a {\it weak lexicographic ideal} as an ideal $I$ of $E$ such $I\ne E$ and it satisfies all conditions (i)---(iii) of the lexicographic ideal, we see that if $I$ is the zero weak lexicographic ideal, then in Theorem \ref{th:6.5}, we can repeat its proof and we obtain $E\cong E/I\cong\Gamma(H,u)$ and $H$ and $E$ are an antilattice po-group and an antilattice effect algebra, respectively, with RDP, and $\langle I\rangle =\{0,1\}$ so that $G=O$, $\langle I\rangle \cong \Gamma(\mathbb Z \lex O,(1,0))$, and $E\cong \Gamma(H\lex O,(u,0))$ holds. Therefore, we have the following result where it is not necessary to assume that $E_0$ and $G$ are nontrivial.

In  Example \ref{ex:3.1}, we have three weak lexicographic ideals, the third one is the zero ideal. Of course, the set of weak lexicographic ideals is linearly ordered with respect to the set theoretical inclusion, see Proposition \ref{pr:3.2}.

\begin{theorem}\label{th:6.6}
Let $E$ be an effect algebra with \RDP\, and let $I$ be a weak lexicographic ideal of $E$. Then there is an Abelian antilattice unital po-group $(H,u)$ with \RDP\, such that $E/I \cong \Gamma(H,u)$ and there is an Abelian directed po-group $G$ with \RDP\, and with $\langle I\rangle \cong \Gamma(\mathbb Z \lex G,(1,0))$ such that $E \cong \Gamma(H\lex G,(u,0))$.

Conversely, if  $E \cong \Gamma(H\lex G,(u,0))$, where $(H,u)$ is an Abelian antilattice unital po-group with \RDP\, and $G$ is an Abelian directed po-group $G$ with \RDP, then $E$ has a weak lexicographic ideal $I$ such that $E/I \cong \Gamma(H,u)$ and $\langle I\rangle \cong \Gamma(\mathbb Z\lex G,(1,0))$.
\end{theorem}

\begin{proof}
For the first part, we repeat all steps of the proof of Theorem \ref{th:6.5}, and for the converse, if $E \cong \Gamma(H\lex G,(u,0))$, then $I=\{0\}\times G^+$ is a weak lexicographic ideal.
\end{proof}

The condition ``$I$ is a prime ideal" in the definition of a weak lexicographic ideal $I$ of $E$ means that $\Gamma(H,u)\cong E/I$ is an antilattice Abelian po-group. If we cancel this property, Theorem \ref{th:6.5} can be reformulated as follows:

\begin{theorem}\label{th:6.7}
Let $E$ be an effect algebra with \RDP\, and let $I\ne E$ be an ideal of $E$ that is strict and retractive. Then there is an Abelian unital po-group $(H,u)$ with \RDP\, such that $E/I \cong \Gamma(H,u)$ and there is an Abelian directed po-group $G$ with \RDP\, and with $\langle I\rangle \cong \Gamma(\mathbb Z \lex G,(1,0))$ such that $E \cong \Gamma(H\lex G,(u,0))$.

Conversely, if $E \cong \Gamma(H\lex G,(u,0))$, where $(H,u)$ is an Abelian unital po-group with \RDP\, and $G$ is an Abelian directed po-group $G$ with \RDP, then $E$ has a strict and retractive ideal $I$ such that $E/I \cong \Gamma(H,u)$ and $\langle I\rangle \cong \Gamma(\mathbb Z\lex G,(1,0))$.

In addition, an effect algebra $E$ with \RDP\, is strong $(H,u)$-perfect if and only if $E$ has a strict and retractive ideal $I$ such that $E/I \cong \Gamma(H,u)$.
\end{theorem}

\begin{proof}
The first two statements have the same proof as that of Theorem \ref{th:6.6}. The third statement follows from the first two ones and Theorem \ref{th:5.1}.
\end{proof}

Let $I$ be a lexicographic ideal of an effect algebra $E$. We say that $E$ is $I$-{\it representable} if $E\cong \Gamma(H\lex G, (u,0))$, where $(H,u)$ is an Abelian unital po-group with RDP such that $E/I \cong \Gamma(H,u)$ and $G$ is a directed po-group with RDP such that $\langle I\rangle \cong \Gamma(\mathbb Z \lex G,(1,0))$;  the existence of $(H,u)$ and $G$ is guaranteed by Theorem \ref{th:6.5}. In view of Theorem \ref{th:6.6} and Theorem \ref{th:6.7}, this notion can be extended also for any weak lexicographic ideal $I$ or for any ideal $I$ that strict and retractive.

We say that a homomorphism $f:E \to F$ (i) has the $\sim$-{\it property} if $f(x)=f(y)$ iff there are $e,g\in \Ker(f):=\{x\in E: f(x)=0\}$ with $e\le x$ and $g\le y$ such that $x-e=y-g$, and (ii) is {\it full} if $f(a)+f(b)$ is defined in $F$, there are $a_1,b_1\in E$ such that $f(a_1)=f(a)$, $f(b_1)=f(b)$ and $a_1+b_1$ is defined in $E$.

For example, if $I$ is a Riesz ideal of $E$  and $f=\pi_I$, then $f$ has the $\sim$-property, \cite[Prop 4.1]{185}, and $\Ker(f)\subseteq I$. We note that if $f$ is surjective and full, the relation $\sim_f$ on $E$ defined by $a\sim_f b$ iff $f(a)=f(b)$ is a congruence on $E$ such that $E/\sim_f$ is an effect algebra, see \cite[Prop 3.3]{DvVe3}.

\begin{lemma}\label{le:6.8}
If a surjective homomorphism $f:E\to F$ has the $\sim$-property, then $f$ is full. Moreover, if $f(x)+f(y)$ exists in $F$, there are $x_1,y_1\in E$ with $f(x_1)=f(x)$, $f(y_1)=f(y)$ and $x_1+y_1$ exists in $E$.

In addition, if $E$ has the \RDP, then $E/\sim_f=E/\Ker(f)$ has also the \RDP.
\end{lemma}

\begin{proof}
If $f$ has the $\sim$-property, then $a\sim_f b$ iff $a\sim_{Ker(f)}b$. Let $f(x)+f(y)$ be defined in $F$. Then $f(x)\le f(y^-)$. Since the classes $a/\sim_f = a/\sim_{Ker(f)}$, by \cite[Prop 4.1]{185}, there is $x_1 \in E$ such that $f(x_1)=f(x)$ and $x_1\le y^-$. Hence, $f$ is full.

Now let $E$ have the RDP. Then $E/\Ker(f)$ has the RDP. But $E/\Ker(f)=E/\sim_f$, so that $E$ has RDP, too.
\end{proof}

In the next theorem we show that the class of lexicographic effect algebras with \RDP\, is closed under homomorphic images.

\begin{theorem}\label{th:6.9}
Let $I$ be a  (weak) lexicographic ideal of an effect algebra with \RDP, and let $f$ be a surjective homomorphism with the $\sim$-property from $E$ onto an effect algebra $F$ with \RDP\, such that $\Ker(f)\subset I$ $(\Ker(f)\subseteq I)$. Then $f(I)$ is a (weak) lexicographic ideal of $F$.

In addition, $F\cong \Gamma(H_1\lex G_1,(u_1,0))$, where $(H_1,u_1)$ and $G_1$ is an Abelian unital po-group and an Abelian directed po-group, respectively, that are homomorphic images of $(H,u)$ and $G$, respectively.
\end{theorem}

\begin{proof}
Let $I$ be a lexicographic ideal of $E$ such that $E$ is $I$-representable.

Take an effect algebra $F$ with RDP and let $f:E\to F$ be a surjective homomorphism with the $\sim$-property. By Lemma \ref{le:6.8}, we can show that $f(I):=\{f(x): x\in I\}$ is an ideal of $F$. Indeed, (i) let $f(x)\le f(y)\in f(I)$. There exists $x_1\in E$ with $f(x_1)=f(x)$ and $x_1 \le y$ and there exist $a,b\in \Ker(f)$ with $a\le x_1$, $b\le x$ such that $x_1-a=x-b$. Then $x_1=(x-b)+a\le y$ which yields $x_1\in I$ and $f(x)=f(x_1)\in f(I)$.
(ii) Let $f(x)+f(y)$ be defined in $F$ for $x,y \in I$. By Lemma \ref{le:6.8}, there are $x_1,y_1\in E$ such that $f(x_1)=f(x)$, $f(y_1)=f(y)$ and $x_1+y_1$ is defined in $E$. Then there are $a,b \in \Ker(f)$ with $a\le x_1$, $b\le x$ such that $x_1-a=x-b$. Then $x-b \in I$ as well as $x_1-a\in I$. In the same way, there are $c,d \in \Ker(f)$ with $c\le y_1$, $d\le y$ such that $y_1-c=y-d$ giving $y_1-d\in I$ and $(x_1-a)+(y_1-c)$ exists in $I$ so that $f(x)+f(y)=f(x_1-a)+f(y_1-c)\in f(I)$.

We assert that $f(I_0(a))=I_0(f(a))$. In fact, due to definition of $I_0(a)$, see (2.1), we see that $f(I_0(a))\subseteq I_0(f(a))$. Conversely, let $f(x)\in I_0(f(a))$. We have $f(x)=f(a_1)+\cdots +f(a_n)$ for $f(a_i)\le f(a)$. By Lemma \ref{le:6.8}, we can assume that $a_i \le a$ and $a_1+\cdots +a_n $ is defined in $E$. Then $f(x)=f(a_1+\cdots+a_n)$. There is $x_1 \le x$ such that $f(x_1)=f(x)$ so that $x_1 \in I_0(a)$ and $f(x)=f(x_1)\in f(I_0(a)).$

Now we show that $f(I)$ is a prime ideal. To show that, we have to prove that if $I_0(f(a))\cap I_0(f(b)) \subseteq f(I)$, then $f(a)\in f(I)$ or $f(b)\in f(I)$. We have $f(I_0(a)\cap I_0(b))\subseteq f(I_0(a))\cap f(I_0(b))= I_0(f(a))\cap I_0(f(b))\subseteq f(I)$.

Choose $x \in I_0(a)\cap I_0(b),$ then $f(x)=f(y)$ for some $y \in I$. There is $x_1\le y$ such that $f(x_1)=f(y)$. There are $e,g\in \Ker(f)$ such that $e\le x_1$, $g\le x$ and $x_1-e=x-g$. On the other hand, there is $z \in \Ker(f)$ with $z \le y$ such that $x_1=y-z$. Then $z \in I$ and $(x-g) +e = x_1 =y-z$ and $x = ((y-z)-e)+g \in I$ while $g \in \Ker(f)\subseteq I$ which entails $I_0(a)\cap I_0(b)\subseteq I$. Since $I$ is prime, $a\in I$ or $b \in I$ so that $f(a)\in f(I)$ or $f(b)\in f(I)$ proving that $f(I)$ is prime.

Now we show that $f(I)$ is strict. Let $f(x)/f(I)< f(y)/f(I)$. There exists $x_1\in I$ such that $f(x_1)/f(I)=f(x)/f(I)$ and $f(x_1)< f(y)$. Then there exists $x_2 \in E$ such that $f(x_2)=f(x_1)$ and $x_2<y$. But $x/I=x_2/I <y/I$ which yields $x<y$.

Finally, we show that (i) $f(I)\ne F$ and (ii) $f(I)\ne \{0\}$ whenever $I$ is lexicographic.  Indeed, if $1 \in f(I)$, there is an element $x\in I$ such that $f(1)=f(x)=1$. There is an element $e\in \Ker(f)\subseteq I$ such that $1-e\le x$, so that $1=(1-e)+e \in I$, absurd. Now if $f(x)=0$ for every $x\in I$. There is an element $x\in I\setminus \Ker(f)$. Then $f(x)=f(0)$. We can find an element $e\in \Ker(f)$ such that $0\le x-e\le 0$, i.e. $x=e$, absurd.

We claim that $f(I)$ is a retractive ideal. Let $\pi_I: E\to E/I$ be the canonical projection and let $\delta_I:E/I\to E$ be a homomorphism such that $\pi_I \circ \delta_I = id_{E/I}$. Let $E_0=\delta_I(E/I)$ be a subalgebra of $E$ that is isomorphic to $E/I$. If we define $\hat f: E/I \to F/f(I)$  by $\hat f(x/I)= f(x)/f(I)$, then $\hat f$ is a well-defined surjective homomorphism such that $\hat f \circ \pi_I= \pi_{f(I)}\circ f$. Set $F_0=f(E_0)$ and let $f_{E_0}$ be the restriction of $f$ onto $E_0$. We define $\delta_{f(I)}: F/f(I) \to F$ via $\delta_{f(I)}(f(x)/f(I)):= f_{E_0}(\delta_I(x/I))$; then $\delta_{f(I)}$ is a well-defined homomorphism such that $\delta_{f(I)}(F/f(I))=F_0$ and $f_{E_0} \circ \delta_I = \delta_{f(I)} \circ \hat f$. Hence,

\begin{eqnarray*}
\pi_{f(I)}\circ \delta_{f(I)}(f(x)/f(I)) &=& \pi_{f(I)} \circ f_{E_0} \circ \delta_I(x/I)\\
&=& \hat f \circ \pi_I \circ \delta_I(x/I)= \hat f(x/I)\\
&=& f(x)/f(I)
\end{eqnarray*}
that proves $f(I)$ is a retractive ideal of $E$.

We have just proved that $f(I)$ is a (weak) lexicographic ideal of $F$. By Theorems \ref{th:6.5}--\ref{th:6.6}, there are Abelian unital po-groups $(H,u)$ and $(H_1,u_1)$ with RDP  such $E/I\cong \Gamma(H,u)$ and $F/f(I)\cong \Gamma(H_1\lex G_1, (u_1,0))$, and there are Abelian directed po-groups $G$ and $G_1$ with RDP and with $\langle I \rangle \cong \Gamma(\mathbb Z\lex G,(1,0))$ and $\langle f(I)\rangle \cong \Gamma(\mathbb Z \lex G_1,(1,0))$ such that $E \cong \Gamma(H\lex G,(u,0))$ and $F=\Gamma(H_1\lex G_1,(u_1,0)$. We define a mapping $f_I: E/I \to F/f(I)$ by $f_I(x/I):=f(x)/f(I)$ $(x \in E)$; it is a well-defined mapping because if $x/I=y/I$, there are $e,g \in I$ with $e\le x$, $g \le y$ and $x-e=y-g$. Then $f(x)-f(e)=f(y)-f(g)$ and $f(e),f(g)\in f(I)$, $f(e)\le f(x)$, $f(g)\le f(y)$ proving $f(x)/f(I)=f(y)/f(I)$. Then $f_I$ is a surjective homomorphism, so that $H_1$ is a surjective homomorphism of $H$. In addition, $I$ and $\phi(I)$ are associative cancellative semigroups satisfying conditions of Birkhoff's Theorem  \cite[Thm XIV.2.1]{Bir}, \cite[Thm II.4]{Fuc}. Then they are positive cones of po-groups $G$ and $G_1$ $(f(G^+)=f(I)=G_1^+)$, respectively, consequently, $G_1$ is a homomorphic image of $G$.
\end{proof}

We note that the condition ``$\Ker(f) \subseteq I$" in Theorem {\rm \ref{th:6.9}} was used only to show that $f(I)$ is a prime ideal of $F$ and  that $f(I)\ne \{0\}$ whenever $\Ker(f)\subset f(I)$. Hence, if we need only that $I$ is strict and retractive, we have the following result.

\begin{theorem}\label{th:6.10}
Let $E=\Gamma(H\lex G,(u,0))$, where $(H,u)$ is an Abelian unital po-group with \RDP\, and $G$ is an Abelian directed po-group with \RDP, and let $F$ be an effect algebra with \RDP. If $f:E\to F$ is a surjective homomorphism with the $\sim$-property, then $F\cong \Gamma(H_1\lex G_1,(u_1,0))$, where $(H_1,u_1)$ and $G_1$ is an Abelian unital po-group and an Abelian directed po-group, respectively, that are homomorphic images of $(H,u)$ and $G$, respectively.
\end{theorem}

\begin{proof}
Due to Theorem \ref{th:6.7}, there is a strict and retractive ideal $I$ of $E$ such that $E/I=\Gamma(H,u)$ and $\langle I \rangle \cong \Gamma(\mathbb Z \lex G,(1,0)).$ Using the proof of Theorem \ref{th:6.9}, we see that $f(I)$ is a strict and retractive ideal of $F$. Applying again Theorem \ref{th:6.7}, we have the desired statement.
\end{proof}


We say that an effect algebra $E$ is a {\it subdirect product} of a system of effect algebras $(E_i: i\in I)$ if there is an injective homomorphism $f: E \to \prod_{i\in I}E_i$ such that $f(a)\le f(b)$ if and only if $a\le b$ ($a,b \in E$), and for every $j \in J$, $\pi_j \circ f$ is a surjective homomorphism from $E$ onto $E_j$, where $\pi_j$ is the projection of $\prod_{i\in I}E_i$ onto $E_j$.

The following subdirect representation theorem of effect algebras with RDP was proved in \cite[Thm 7.2]{177}:

\begin{theorem}\label{th:6.11}
Every effect algebra $E$ with the \RDP\, is a subdirect product of antilattice effect algebras with the \RDP, and all existing meets and joins in $E$ are preserved in the subdirect product.
\end{theorem}

The proof of this theorem shoved that $\mathcal P(E)$ has a lot of proper prime ideals of $E$ so that
$$\bigcap \{P: P\in \mathcal P(E)\setminus\{E\}\}=\{0\}. \eqno(6.1)
$$

Combining Theorem \ref{th:6.11} with the previous representation theorems, we can prove the following subdirect product representability result:

\begin{theorem}\label{th:6.12}
Let $E=\Gamma(H\lex G,(u,0))$, where $(H,u)$ is an Abelian unital po-group with \RDP\, and $G$ is an Abelian directed po-group with \RDP. Then $E$ is a subdirect product of a family $(\Gamma(H_i\lex G_i,(u_i,0)))$ of antilattice effect algebras with \RDP, where each $(H_i,u_i)$ is a unital Abelian po-group with \RDP\, and $G_i$ is a directed Abelian po-group with \RDP.
\end{theorem}

\begin{proof}
By Theorem \ref{th:6.7}, the ideal $I=\{0\}\times G^+$ is strict and retractive. Let $P$ be a prime proper ideal of $E$; due to (6.1), we can find a lot of prime ideals. The mapping $f_P: E \to E/P$ is a surjective and has the $\sim$-property. In addition, $E/P$ is an antilattice. By Theorem \ref{th:6.11}, $E/P\cong \Gamma(H_P\lex G_P,(u_p,0))$ for some Abelian unital po-group $(H_P,u_P)$ with RDP and for some Abelian directed po-group $G_P$ with RDP. Due to (6.1), $E$ is a subdirect product of the family $(\Gamma(H_P\lex G_P,(u_P,0)): P \in \mathcal P(E)\setminus \{E\})$.
\end{proof}

\section{Conclusion}

In the paper we have found conditions when an effect algebra with RDP is of the form of lexicographic product $\Gamma(H\lex G,(u,0))$, where $(H,u)$ is an Abelian unital po-group with RDP and $G$ is an Abelian directed po-group with RDP. We have shown that the crucial notion was a strong $(H,u)$-perfect effect algebra which can be split into slices $(E_t: t \in  [0,u]_H)$ indexed by the elements of $\Gamma(H,u)$ with some natural properties. This notion  established in Theorem \ref{th:5.1} a representation theorem and allowed us to show that the category of strong $(H,u)$-perfect effect algebras is categorically equivalent to the category od directed po-groups with RDP, Theorem \ref{th:5.5}.

Another important notion was a retractive ideal a lexicographic ideal. Also in this case we have proved representation theorems of local effect algebras and cases when $(H,u)$ is an antilattice po-group, see Theorems \ref{th:6.4}--\ref{th:6.5}. Finally, we have showed that every lexicographic effect is a subdirect product of antilattice lexicographic effect algebras with RDP, Theorem \ref{th:6.12}.

The paper showed an important class of effect algebras connected with lexicographic product of two Abelian po-groups with RDP. It solved some interesting questions and stimulated a new research on this topic, mainly on lexicographic product of pseudo effect algebras.


\begin{thebibliography}{DvGr2}

\bibitem[Bir]{Bir}
G. Birkhoff, {\it ``Lattice Theory"}, Amer.
Math. Soc. Coll. Publ., Vol. {\bf 25},
Providence, Rhode Island, 1967.

\bibitem[CiTo]{CiTo}
R. Cignoli, A. Torrens, {\it Retractive MV-algebras}, Mathware Soft Comput. {\bf 2} (1995), 157--165.

\bibitem[DFL]{DFL}
D. Diaconescu, T. Flaminio, I. Leu\c{s}tean, {\it Lexicographic MV-algebras and lexicographic states}, Fuzzy Sets and System {\bf 244}  (2014), 63--85. DOI         10.1016/j.fss.2014.02.010

\bibitem[DiLe1]{DiLe1}
A. Di Nola, A. Lettieri, {\it Perfect MV-algebras are categorical
equivalent to abelian $\ell$-groups}, Studia Logica {\bf 53} (1994),
417--432.

\bibitem[DiLe2]{DiLe2}
A. Di Nola, A. Lettieri, {\it Coproduct MV-algebras, nonstandard reals and Riesz spaces}, J. Algebra {\bf 185} (1996), 605--620.

\bibitem[Dvu0]{151}
A. Dvure\v censkij,    {\it  Pseudo MV-algebras
are intervals in $\ell$-groups}, J. Austral. Math. Soc. {\bf 72} (2002), 427--445.

\bibitem[Dvu1]{185}
A. Dvure\v censkij,    {\it Ideals of pseudo-effect algebras
and their applications}, Tatra Mt. Math. Publ.    {\bf 27} (2003),
45--65.

\bibitem[Dvu2]{177}
A. Dvure\v censkij, {\it  Perfect effect algebras are
categorically equivalent with Abelian interpolation po-groups},
J. Austral. Math. Soc. {\bf 82} (2007), 183--207.

\bibitem[Dvu3]{Dv08}   A. Dvure\v{c}enskij,  {\it On $n$-perfect GMV-algebras},
J. Algebra {\bf 319} (2008), 4921--4946.


\bibitem[Dvu4]{264}
A. Dvure\v censkij,    {\it $\mathbb H$-perfect pseudo MV-algebras and their representations},  Math. Slovaca   {\bf } http://arxiv.org/abs/1304.0743

\bibitem[Dvu5]{275}
A. Dvure\v censkij, {\it Pseudo MV-algebras and lexicographic product}, http://arxiv.org/abs/1406.2339

\bibitem[DvKo]{DvKo}
A. Dvure\v censkij, M.  Kola\v{r}\'\i k, {\it Lexicographic product vs $\mathbb Q$-perfect and $\mathbb H$-perfect pseudo effect algebras}, Soft Computing {\bf 17} (2014), 1041--1053.  DOI: 10.1007/s00500-014-1228-6

\bibitem[DvKr]{DvKr}
A. Dvure\v censkij, J. Kr\v{n}\'avek,  {\it The lexicographic product of po-groups and $n$-perfect pseudo effect algebras}, Inter. J. Theor. Phys. {\bf 52} (2013), 2760--2772. DOI:10.1007/s10773-013-1568-5

\bibitem[DvPu]{DvPu}
A. Dvure\v censkij, S. Pulmannov\'a, {\it ``New Trends in Quantum
Structures"}, Kluwer Acad. Publ., Dordrecht, Ister Science,
Bratislava, 2000.

\bibitem[DvVe1]{DvVe1} A. Dvure\v censkij, T. Vetterlein,   {\it
Pseudoeffect algebras. I. Basic properties},  Inter. J. Theor.
Phys. {\bf 40} (2001), 685--701.

\bibitem[DvVe2]{DvVe2} A. Dvure\v censkij, T. Vetterlein,   {\it
Pseudoeffect algebras. II. Group representation}, Inter. J. Theor.
Phys. {\bf 40} (2001), 703--726.

\bibitem[DvVe3]{DvVe3}
A. Dvure\v censkij, T. Vetterlein,   {\it
Congruences and states on pseudo-effect algebras,} Found. Phys.
Letters {\bf 14} (2001), 425--446.

\bibitem[DvVe4]{DvVe4}
A. Dvure\v censkij, T. Vetterlein,    {\it
Archimedeanicity and the MacNeill completion of pseudoeffect
algebras and po-groups}, Algebra Universalis {\bf 50} (2003),
207--230.



\bibitem[DXY]{DXY}
A. Dvure\v censkij, Y. Xie, Aili  Yang,  {\it Discrete $(n+1)$-valued states and $n$-perfect  pseudo-effect algebras,} Soft Computing {\bf 17} (2013), 1537–-1552. DOI: 10.1007/s00500-013-1001-2

\bibitem[FoBe]{FoBe}  D.J. Foulis, M.K. Bennett,
{\it  Effect algebras and unsharp quantum logics}, Found. Phys. {\bf
24} (1994), 1331--1352.

\bibitem[Fuc]{Fuc} L. Fuchs, {\it ``Partially Ordered Algebraic Systems"}, Pergamon Press, Oxford-New York, 1963.

\bibitem[GeIo]{GeIo} G. Georgescu, A. Iorgulescu, {\it
Pseudo-MV algebras}, Multiple Val. Logic {\bf 6} (2001), 95--135.


\bibitem[Gla]{Gla}
A.M.W. Glass {\it ``Partially Ordered Groups"}, World Scientific,
Singapore, 1999.

\bibitem[Go]{Goo}
 K.R. Goodearl,
{\it ``Partially Ordered Abelian Groups with Interpolation"},
 Math. Surveys and Monographs No. 20, Amer. Math. Soc.,
 Providence, Rhode Island, 1986.

\bibitem[LuZa]{LuZa}
W.A.J. Luxemburg, A.C. Zaanen,
{\it ``Riesz Spaces I"},
 North-Holland, Amsterdam, London, 1971.


\bibitem[MaL]{MaL}
S. Mac Lane, {\it ``Categories for
the Working Mathematician"}, Springer-Verlag,   New York, Heidelberg,
Berlin, 1971.

\bibitem[Mun]{Mun}
D. Mundici, {\it   Interpretation of AF $C^*$-algebras in \L
ukasiewicz sentential calculus}, J. Funct. Anal. {\bf 65} (1986),
15--63.

\bibitem[Pul]{Pul}
S. Pulmannov\'a,   {\it Compatibility and decomposition of effects}, J. Math. Phys. {\bf 43} (2002), 2817--2830.

\bibitem[Rav]{Rav} K. Ravindran,
{\it  On a structure theory of effect algebras}, PhD thesis, Kansas
State Univ., Manhattan, Kansas, 1996.

\end{thebibliography}
\end{document}